\documentclass[letterpaper, 10 pt, conference]{ieeeconf}  
\IEEEoverridecommandlockouts                              

\usepackage{wrapfig}
\usepackage{amsmath,amssymb,amsfonts,amsthm}
\usepackage{graphicx}
\usepackage{algorithm}
\usepackage{algcompatible}
\usepackage{algpseudocode}
\usepackage{textcomp}
\usepackage{mathrsfs}
\usepackage{subcaption}
\usepackage{xcolor}
\allowdisplaybreaks[4] 

\usepackage[noadjust]{cite}

\newcommand{\res}[1]{\textcolor{black}{#1}}

\newtheorem{theorem}{Theorem}
\newtheorem{definition}[theorem]{Definition}

\newtheorem{corollary}[theorem]{Corollary}
\newtheorem{remark}{Remark}
\newtheorem{assumption}{Assumption}
\newtheorem{proposition}[theorem]{Proposition}
\newtheorem{lemma}[theorem]{Lemma}
\newtheorem{problem}{Problem}

\usepackage{hyperref}

\newenvironment{rese}
  {\color{black}}
  {}

\title{\LARGE \bf Identification with Orthogonal Basis Functions: Convergence Speed, Asymptotic Bias, and Rate-Optimal Pole Selection
\thanks{This paper is supported in part by the National Natural Science Foundation of China (NSFC) under Grant Nos. 62461160313, 62273196, 62192752, and the BNRist project (No. BNR2024TD03003), in part by the Hong Kong RGC under the projects CityU 11207823, CityU 11206024, and N-CityU133/24. The authors would like to thank Prof. Rebing Wu for providing the inverted pendulum platform that supported additional experimental validation related to this work.}
}
\author{Jiayun Li$^{1}$, Yiwen Lu$^{1}$, Yilin Mo$^{1}$ and Jie Chen$^{2}$%
\thanks{$^{1}$Jiayun Li, Yiwen Lu, and Yilin Mo are with the Department of Automation and BNRist, Tsinghua University, Beijing, P.R. China. Emails: \texttt{\{lijiayun22, luyw20\}@mails.tsinghua.edu.cn, ylmo@tsinghua.edu.cn}.}%
\thanks{$^{2}$Jie Chen is with the Department of Electrical Engineering, City University of Hong Kong, P.R. China. Email: \texttt{jichen@cityu.edu.hk}.}%
}
\begin{document}

\maketitle
\thispagestyle{empty}
\pagestyle{empty}

\begin{abstract}
This paper is concerned with performance analysis and pole selection problem in identifying linear time-invariant (LTI) systems using orthogonal basis functions (OBFs), a system identification approach that consists of solving least-squares problems and selecting poles within the OBFs. Specifically, we analyze the convergence properties and asymptotic bias of the OBF algorithm, and propose a pole selection algorithm that robustly minimizes the worst-case identification bias, with the bias measured under the $\mathcal{H}_2$ error criterion. Our results include an analytical expression for the convergence rate and an explicit bound on the asymptotic identification bias, which depends on both the true system poles and the preselected model poles. Furthermore, we demonstrate that the pole selection algorithm is asymptotically optimal, achieving the fundamental lower bound on the identification bias. The algorithm explicitly determines the model poles as the so-called Tsuji points, and the asymptotic identification bias decreases exponentially with the number of basis functions, with the rate of decrease governed by the hyperbolic Chebyshev constant. Numerical experiments validate the derived bounds and demonstrate the effectiveness of the proposed pole selection algorithm.


\end{abstract}


\section{Introduction}\label{sec:introduction}
System identification has long been a focus in the design of control systems, and more broadly in the modeling of dynamical systems \cite{Ljung1998}, statistical time series analysis \cite{Madsen2008}, and in system realization theory \cite{Picci2015}, with early inspirations going back to the works of Zadeh \cite{Zadeh1962} and Kalman \cite{Kalman1965} in the 1960s. Traditional identification methods usually address this problem in the asymptotic sense, where the number of data samples tends to infinity. The central issue is whether the identified mathematical model asymptotically approaches the true physical system, depending on {\em a priori} information available about the system and {\em a posteriori} data acquired experimentally, and, if so, whether a bias arises between the model and the system. Classical identification algorithms, such as the Ho-Kalman algorithm~\cite{ho1966effective} and the multivariable output-error state space (MOESP) approach~\cite{Verhaegen1994}, are known to be asymptotically unbiased given infinite samples under certain conditions~\cite{Ljung1998, Bauer2000, Van1996}. However, these algorithms are typically nonlinear (involving nonlinear operations such as SVD or matrix inversion of the Hankel matrices) with respect to the collected data and may exhibit high sensitivity to noise when only finite-length samples are available~\cite{chiuso_ill-conditioning_2004, hachicha2014n4sid}, resulting in potential high signal-to-noise ratio requirement and thus high sample complexity. \res{This behavior is also illustrated in a numerical example in Fig.~\ref{fig:online_error}.}

It is intuitively plausible that the availability of {\em a priori} information could improve the data efficiency of identification algorithms. In this vein, several methods~\cite{fu1993optimum, Wahlberg1994system, Ninness1999} seek to approximate a system's transfer function using a linear combination of orthogonal basis functions (OBFs). A least-squares (LS) algorithm is leveraged to identify the coefficients of OBFs, which is linear with respect to the data. In contrast, the set of the OBFs, parametrized by their poles, are preselected based on \emph{a priori} information, eliminating the need for nonlinear identification algorithms to estimate system poles. Common choices for OBFs include the Laguerre~\cite{wahlberg_system_1991} and the Kautz functions~\cite{Wahlberg1994system}. Furthermore, Van den Hof et al.~\cite{van_den_hof_system_1995} propose the generalized OBF (GOBF) and derive upper bounds on the asymptotic identification bias as sample size approaches infinity~\cite{van_den_hof_system_1995,heuberger_generalized_1995, ninness_asymptotic_1996}. Ninness and Gustafsson~\cite{ninness_unifying_1997} develop a unifying construction of the OBFs, with which they show that the identification result converges in the mean-square sense at a rate of $O(N^{-1})$ with respect to the sample size $N$. Additionally, Ninness and Gomez \cite{ninness_asymptotic_1996} quantify the asymptotic bias in identification using OBF methods, which can vary significantly depending on the choice of OBFs. Consequently, the selection of poles in the OBFs plays a crucial role in the quality of the identified model, directly influencing the bias of the identification algorithm.


To facilitate the selection of OBF poles, one line of work aims to address this difficulty from a \emph{robustness} perspective, ensuring any system consistent with the given \emph{a priori} information can be accurately approximated by the OBFs with the selected poles. For example, Oliveira e Silva~\cite[Chapter 11]{van_den_hof_system_2005} introduces a pole selection algorithm for the OBFs by numerically solving a nonlinear minimax optimization problem, based on the pole region of the true system, i.e., a closed set in the complex plane that encompasses all poles of the true system, as specified by \emph{a priori} information. However, this problem involves a non-convex semi-infinite programming, which may suffer from drastic growth in computational complexity with the number of poles~\cite[Chapter 11]{van_den_hof_system_2005}. Toth et al.~\cite{toth_asymptotically_2009} propose a clustering algorithm for selecting OBF poles in the context of linear parameter-varying systems. However, the algorithm requires an ensemble of sampled system poles, a requirement that renders the algorithm's performance heavily dependent on the accuracy of the sampled poles. 


There are also less conventional methods that seek to select the OBF poles adaptively, such as particle swarm optimization~\cite{reginato_selecting_2007} and genetic optimization~\cite{sabatini_hybrid_2000}. Other approaches include solving an empirical Bayes problem~\cite{chen_regularized_2015, darwish2015perspectives, darwish_bayesian_2017} and applying a greedy algorithm using frequency-domain data~\cite{mi_adaptive_2021, mi_frequency-domain_2012}. However, these methods often suffer from high computational complexity and may become trapped in local optima.
Yet some other contemporary studies choose to bypass pole specification, in which 
pole locations are not preselected based on {\em a priori} information, but are determined from {\em a posteriori} data. This is the case with the $\mathcal{H}_\infty$ identification method (see, e.g., ~\cite{chen2000}), whereby the system model and henceforth its poles are synthesized through rational approximation of analytic functions under a worst-case $\mathcal{H}_\infty$ criterion. Algorithms in this category are also computationally demanding.

This paper enhances the OBF approach by proposing an optimal pole selection algorithm to \emph{robustly} minimize the worst-case identification bias among a particular class of systems. To this end, we first analyze, the convergence properties and the asymptotic bias of OBF algorithms (i.e., with prescribed poles and OBF bases). We then solve an asymptotic pole selection problem. We show that the so-called Tsuji points~\cite{tsuji1950metrical} asymptotically
achieve the minimal identification bias possible and hence are asymptotically optimal poles. These points can be determined by solving a maximization problem on the boundary of the pole region, ridding of the need to solve a significantly more complex minimax problem, such as the one proposed by Oliveira e Silva~\cite[Chapter 11]{van_den_hof_system_2005}. We also present an algorithm to compute a set of near-optimal initial points for the maximization problem, thus alleviating the issue of local optimality. The resulting system identification bias asymptotically achieves the fundamental lower bound on the worst-case identification bias. Our main contributions can be summarized as follows:
\begin{enumerate}
    \item With preselected bases and corresponding poles, we show that the identification error of the OBF methods measured under the $\mathcal H_2$ norm converges to an asymptotic bias almost surely at a rate of $O(N^{-0.5+\epsilon})$, for any $\epsilon>0$. We also derive an upper bound on the identification bias, which depends on the discrepancy between the true system poles and the preselected poles in the OBF basis functions. When system poles are known to lie within a specific region $\mathcal D$ within the unit disk in the complex plane, we establish a fundamental lower bound on the asymptotic bias, showing that the worst case bias decreases at a rate of at most $O(\tau(\mathcal D)^q)$, where $\tau(\mathcal D)<1$ is the hyperbolic Chebyshev constant of $\mathcal D$ and $q$ denotes the number of bases.
    \item The search for optimal poles, which minimize the asymptotic bias against all possible systems consistent with the \emph{a priori} information, naturally leads to a minimax problem. We propose to replace this minimax problem with a maximization problem on the boundary of the pole region, where the optimal solutions are the Tsuji points. More specifically, we show that the Tsuji points asymptotically achieve the fundamental lower bound on the worst-case identification bias as the number of bases tends to infinity.
    \item By a further analysis, the fundamental lower bound on the identification bias indicates that the number of data length required for identification of $n$-dimensional systems, known as the sample complexity, grows asymptotically at a rate of $O(\tau(\mathcal D)^{-n})$. This points to the intrinsic difficulty in identifying a system when its poles are unknown, especially for high-order systems.
\end{enumerate}

\textit{Paper structure:} In Section~\ref{sec:preliminary}, we introduce relevant concepts and results in complex analysis. Section~\ref{sec:problem_formulation} formulates the system identification problem and briefly reviews the OBF method. Following this, Section~\ref{sec:main_results} analyzes the performance of the OBF method, establishing an almost sure convergence rate with respect to the sample size and deriving an upper bound for the asymptotic identification bias. Section~\ref{sec:pole_allocation_algorithm} derives a fundamental lower bound on the worst-case asymptotic bias and proposes the Tsuji pole selection algorithm, along with an analysis of its performance. In Section~\ref{sec:fundamental_limit}, a fundamental limit on the sample complexity of identifying the system poles is derived. Section~\ref{sec:simulation} provides numerical results to demonstrate the efficacy of the algorithm. Section~\ref{sec:conclusion} concludes the paper.

\textit{Notations:} The symbol $\mathbf{I}_n$ denotes the $n$-dimensional unit matrix, and $\mathbf{1}_n$ is an $n$-dimensional column vector with entries all equal to $1$. Similarly, $\boldsymbol{0}_n$ denotes the $n$-dimensional zero vector. $A^\top$ denotes the matrix transpose and $A^H$ the conjugate transpose of $A$. $\|A\|$ denotes the $2$-norm for a vector $A$ and Frobenius norm for a matrix $A$. For a vector $\boldsymbol{x}$, $\text{diag}(\boldsymbol{x})$ denotes the diagonal square matrix with elements of $\boldsymbol{x}$ on its main diagonal. Moreover, $\mathbb{D}$ denotes the open unit disk on the complex plane, i.e., $\mathbb{D}\triangleq\{z\in\mathbb{C}:|z|<1\}$. We say that $|f(k)|\sim O(g(k))$ for $g(k)>0$ if there exists a constant $M>0$, such that $\lim_{k\to\infty}|f(k)|/g(k)\leq M$ for all $k=1, 2, \cdots$. The extended expectation $\mathbb{\bar E}$ is defined as $\bar{\mathbb E}(\phi(t))=\lim_{t\to\infty}\mathbb E(\phi(t))$ for a proper $\phi$.

\section{Mathematical Preliminary}\label{sec:preliminary}
In this section, we introduce some concepts and results in \textit{complex analysis}, which will be used in the sequel.

We begin with the \textit{pseudohyperbolic metric}, \res{a classical distance notion in complex analysis
(see, e.g.,~\cite{KIRSCH2005243}), defined as}
\begin{equation}\label{eq:pseudohyperbolic_metric}
    [z, \mu]_{\mathrm{h}}\triangleq \left|\frac{z-\mu}{1-\bar\mu z}\right|
\end{equation}
for any $z, \mu$ inside the open unit disk. Next, we introduce the notion of \textit{hyperbolic Chebyshev constant} of a closed region $\mathcal{D}$ inside the open unit disk $\mathbb{D}$.
\begin{definition}[(Finite) Hyperbolic Chebyshev constant (\res{following equation~(55) in}~\cite{KIRSCH2005243}), hyperbolic Chebyshev points]\label{def:chebyshev_constant}
    Let $q$ be an integer and $\mathcal{D}$ a closed subset of the open unit disk $\mathbb{D}$. The finite hyperbolic Chebyshev constant of $\mathcal{D}$ is defined as
    \begin{equation}\label{eq:finite_chebyshev_constant}
        \tau_q(\mathcal{D})\triangleq \min_{\mu_1, \cdots, \mu_q\in\mathcal{D}}\max_{z\in\mathcal{D}}\left(\prod_{k=1}^q [z, \mu_k]_{\mathrm{h}}\right)^{1/q},
    \end{equation}
    and the points $\mu_1, \mu_2, \cdots, \mu_q$ that achieve the minimum are called the hyperbolic Chebyshev points of $\mathcal{D}$.
The limit
    \begin{equation}\label{eq:chebyshev_constant}
        \tau(\mathcal{D})\triangleq\lim_{q \rightarrow \infty} \min_{\mu_1, \cdots, \mu_q\in\mathcal{D}} \max_{z \in \mathcal{D}}\left(\prod_{k=1}^q [z, \mu_k]_{\mathrm{h}}\right)^{1 / q},
    \end{equation}
which always exists, is referred to as the hyperbolic Chebyshev constant of the region $\mathcal{D}$.
\end{definition}

Similarly, we may define the hyperbolic transfinite diameter and the Tsuji points of a region $\mathcal{D}$.
\begin{definition}[Hyperbolic transfinite diameter, Tsuji points \res{(following equation~(54) in~}\cite{KIRSCH2005243})]\label{def:transfinite_diameter_tsuji_point}
Let $q$ be an integer and $\mathcal{D}$ a closed subset of the open unit disk $\mathbb{D}$. Define 
    \begin{equation}\label{eq:finite_transfinite_diameter}
			d_q(\mathcal{D}):= \max _{z_1, \ldots, z_q \in \mathcal{D}}\left\{\prod_{1 \leqslant j<k \leqslant q}\left[z_j,z_k\right]_{\mathrm{h}}\right\}^{1/\binom{q}{2}}
    \end{equation}
\res{Here $\binom{q}{2}=q(q-1)/2$ denotes the number of distinct pairs among the $q$ points.} The points $z_{k}, k=1, \cdots, q$ that attain the maximum are called the $q$-th Tsuji points of $\mathcal{D}$.

Moreover, the following limit exists and is referred to as the hyperbolic transfinite diameter of $\mathcal{D}$:
    \begin{equation}\label{eq:transfinite_diameter}
			d(\mathcal{D}):=\lim _{q \rightarrow \infty} \max _{z_1, \ldots, z_q \in \mathcal{D}}\left\{\prod_{1 \leqslant j<k \leqslant q}\left[z_j,z_k\right]_{\mathrm{h}}\right\}^{1/\binom{q}{2}}.
    \end{equation}
\end{definition}


The hyperbolic Chebyshev constant and the hyperbolic transfinite diameter are known to coincide: 
\begin{proposition}[\res{see page~278 in}~\cite{KIRSCH2005243}]\label{prop:tsuji_chebyshev_relation}
    For any closed set $\mathcal{D}\subset \mathbb{D}$,
    \begin{equation}
        \tau(\mathcal{D})=d(\mathcal{D}).
    \end{equation}
    Let $\mathbb{\bar D}_\rho$ denote the closed disk with radius $\rho<1$, i.e., $\mathbb{\bar D}_\rho\triangleq \{z:|z|\leq \rho\}$. If $\mathcal{D}\subset \mathbb{\bar D}_\rho$, then
    \begin{equation}
        \tau(\mathcal{D})=d(\mathcal{D})\leq \rho.
    \end{equation}
\end{proposition}

For certain special sets $\mathcal{D}$, their hyperbolic Chebyshev constants can be found analytically. 
We summarize some of these results in the following lemma, whose proofs are relegated to 
Appendix~\ref{append:hyperbolic_chebyshev_constant}.

\begin{lemma}[Hyperbolic Chebyshev constant of disks and intervals]\label{lemma:hyperbolic_chebyshev_constant}
   Below are some special sets $\mathcal{D}$ whose hyperbolic Chebyshev constants can be computed analytically:
    \begin{enumerate}
        \item The hyperbolic Chebyshev constant of \textbf{a disk} $\mathbb{\bar D}_\rho$ is
        \[\tau(\mathbb{\bar D}_\rho)=\rho.\]\label{item:disk}
        \item The hyperbolic Chebyshev constant of a \textbf{real interval} $[\rho_1, \rho_2], -1<\rho_1<\rho_2<1$ is\label{item:real_interval}
        \[\tau([\rho_1, \rho_2])=\exp \left\{-\frac{\pi}{2} \frac{K\left(\sqrt{1-\tilde\rho^2}\right)}{K(\tilde\rho)}\right\},\]
        where
        \[\tilde \rho=\frac{\rho_2-\rho_1}{1-\rho_1\rho_2}, K(\rho)\triangleq\int_0^1 \frac{d x}{\sqrt{\left(1-x^2\right)\left(1-\rho^2 x^2\right)}}.\]
    \end{enumerate}
\end{lemma}

\section{Problem Formulation}\label{sec:problem_formulation}
Consider a 
stable LTI system with $m$ inputs and $p$ outputs:
\begin{equation}\label{eq:true_system}
    y_t=G(z)u_t+H(z)e_t =G(z)u_t + v_t,
\end{equation}
where 
$\{u_t\}$ is a quasistationary signal and $\{e_t\}$ is an i.i.d. stochastic vector process with zero mean and unit covariance. The variable $z$ represents the time-shift operator, i.e., $zu_t=u_{t+1}$. 
We make the following assumptions about the system:
\begin{assumption}\label{assump:system_stability} 
\noindent 
    \begin{enumerate}
        \item $G(z)$ and $H(z)$ are both strictly stable transfer functions; that is, all the poles of $G(z)$ and $H(z)$ are in $\mathbb{D}$. Additionally, $G(z)$ has relative degree at least $1$.
        
        \item The system~\eqref{eq:true_system} is controlled by a linear controller and is closed-loop stable. The controller can be modeled as another LTI system:
        \begin{align}\label{eq:stabilizing_controller}
            u_t=G_u(z)y_t+H_u(z)\epsilon_t,
        \end{align}
        where $\{\epsilon_t\}$ is an i.i.d. stochastic vector process with zero mean and unit covariance, serving as probing noise and is assumed to be mutually independent with $\{e_t\}$. \res{The controller transfer functions $G_u(z)$ and $H_u(z)$, as well as the realization of the probing signal $\epsilon_t$, are known and available.}\label{assump:control_input}
    \end{enumerate}
\end{assumption}
\vspace{-\baselineskip}%
\res{\begin{assumption}[Persistent excitation]\label{assump:input}
The externally injected probing component
\(
H_u(z)\epsilon_t
\)
is persistently exciting of sufficiently high order.
\end{assumption}}
\begin{rese}
Assumption~\ref{assump:system_stability} and~\ref{assump:input} implies that the resulting control input $u_t$ is persistently exciting of sufficiently high order.
\begin{remark}
Under Assumptions~\ref{assump:system_stability} and~\ref{assump:input},
the plant transfer function $G(z)$ is uniquely identifiable from closed-loop data.
The noise model $H(z)$ is not explicitly identified in this work, and its uniqueness
is not relevant to the subsequent analysis.
\end{remark}
\end{rese}


In this paper, we focus on applying the OBF method to identify the system~\eqref{eq:true_system}.  Specifically, we need to approximate $G(z)$ with $\check G(z)$, which is a linear combination of a set of predefined basis functions $V_k(z)$:
\begin{equation}
    \check G(z)=\sum_{k=1}^{q} \check R_k V_k(z).
\end{equation}
The coefficient matrices in the linear combinations are, in general, complex-valued matrices and satisfy $\check R_k\in\mathbb C^{p\times m}$.  
On the other hand, the preselected basis functions $V_k(z)$, characterized by a group of $q$ poles $\{\mu_1, \cdots, \mu_q\}$, can be chosen as:
\begin{equation}\label{eq:Vk_basis}
    V_k(z)=\frac{1}{z-\mu_k}, k=1, \cdots, q,
\end{equation}
where $\mu_k$ are distinct from each other. These bases can be further orthonormalized using standard procedures, to obtain the \emph{unified construction} of OBFs~\cite{ninness_unifying_1997} as follows:
\begin{equation}\label{eq:poles_induced_obfs}
     \mathcal V_k(z)=\frac{\sqrt{1-|\mu_k|^2}}{z-\mu_k}\prod_{l=1}^{k-1}\frac{1-\bar\mu_l z}{z-\mu_l}, k=1, \cdots, q.
\end{equation}
This unified construction includes a wide variety of commonly used OBFs, such as Laguerre bases, Kautz bases, and generalized OBF~\cite{van_den_hof_system_1995}. The readers can refer to~\cite{van_den_hof_system_2005} for more details.

\begin{remark}
    Although the orthonormalization procedure may help with the numerical stability of the least-squares algorithm, it does not alter the span of the basis functions, the identification error of the least-squares problem, or the asymptotic approximation bias of the identification discussed later.
\end{remark}
Owing to the equivalence of the span between $V_k$ in~\eqref{eq:Vk_basis} and the unified construction $\mathcal V_k$ in~\eqref{eq:poles_induced_obfs}, in the rest of the paper, we slightly abuse the definition by adopting the bases $V_k$ as the preselected bases in the OBF method for simplicity. Then, the OBF identification problem can be stated as follows:
\begin{problem}[System identification using preselected $V_k$]\label{prob:system_identification}
    \begin{equation}
        \min_{\check R_1, \cdots, \check R_q}\left\|G(z)-\sum_{k=1}^q \frac{\check R_k}{z-\mu_k}\right\|_2.
    \end{equation}
\end{problem}

To find a near-optimal solution for Problem~\ref{prob:system_identification} in a data-driven manner given $N$ sample pairs $u_{1:N}=\{u_1, \cdots, u_N\}, y_{1:N}=\{y_{1}, \cdots, y_{N}\}$, we solve the following least-squares regression~\cite[Chapter 4]{van_den_hof_system_2005}. \res{Under Assumption~\ref{assump:input}, the resulting control input $u_t$
is persistently exciting of sufficiently high order, which guarantees
that the least-squares problem~\eqref{eq:finite_least_squares_1} is
well-defined after a finite warm-up period:}
\begin{align}\label{eq:finite_least_squares_1}
    &\check R_1(N), \cdots, \check R_q(N)\nonumber \\
    &=\arg\min_{\check R_1, \cdots, \check R_{q}}\frac{1}{N}\sum_{t=1}^N \left\|y_t-\sum_{k=1}^{q} \frac{\check R_k}{z-\mu_k}u_t\right\|^2,
\end{align}
where $\|\cdot\|$ denotes the $2$-norm of a vector.  Then, the identified model using $N$ samples is given by \[\check G_N(z)=\sum_{k=1}^q \frac{\check R_k(N)}{z-\mu_k}.\]


\section{Convergence and Asymptotic Bias}\label{sec:main_results}
The overarching goal of this section is to analyze the convergence property of the least-squares algorithm in solving~\eqref{eq:finite_least_squares_1} in the almost sure sense as the number of samples tends to infinity, and to quantify the asymptotic bias using the bases $1/(z-\mu_k)$.


\subsection{Convergence Analysis}\label{subsec:optimal_system_approximation_bias}
Define the solution to the following expected least-squares problem as:
\begin{equation}
	\left(\check R_1^*, \cdots, \check R_q^*\right)=\arg\min_{\check R_1, \cdots, \check R_q}\mathbb{\bar E}\left\|y_t-\sum_{k=1}^q\frac{\check R_k}{z-\mu_k}u_t\right\|^2,\label{eq:expected_least_squares_problem}
\end{equation}
where for a vector (or matrix) valued random process $\{X_t\}$, we denote its limit expected value as $\mathbb{\bar E}(X_t)=\lim_{t\to\infty} \mathbb E(X_t)$~\cite{van_den_hof_system_1995}. The corresponding transfer function is given as:
\begin{equation}\label{eq:asymptotic_result}
    \check G^*(z)=\sum_{k=1}^q \frac{\check R_k^*}{z-\mu_k}.
\end{equation}

The following theorem establishes the property that the identified system $\check G_N(z)$ by solving the least-squares problem~\eqref{eq:finite_least_squares_1} with $N$ samples converges to $\check G^*(z)$ almost surely.
\begin{theorem}[Almost sure convergence]\label{thm:convergence}
    The $\mathcal H_2$ norm of the error between the identified model $\check G_N(z)$ using $N$ samples and the asymptotic model $\check G^*(z)$ satisfies
    \begin{equation}\label{eq:gap_convergence}
        \lim_{N\to\infty}\frac{\|\check G_N(z)-\check G^*(z)\|_{2}}{N^{-0.5+\epsilon}}=0 \ a.s.,
    \end{equation}
    for all $\epsilon>0$.
\end{theorem}
\begin{proof}
    Let a state-space realization of the system in~\eqref{eq:true_system} be given by~\cite[Chapter 2]{van_den_hof_system_2005}:
    \begin{align}
        x_{t+1}=Ax_t+Bu_t+B_e e_{t+1},\quad y_t=Cx_t,
    \end{align}
    and let a state-space realization of the input dynamics~\eqref{eq:stabilizing_controller} be
    \begin{align}
        \varphi_{t}=A_\varphi\varphi_{t-1}+B_{\varphi y} y_t+B_{\varphi \epsilon}\epsilon_t, \quad u_t=C_{\varphi}\varphi_t.
    \end{align}

    Then, the overall system incorporating the true system $G(z)$, the input dynamics, and the identified model can be collectively represented by the augmented system
    \begin{align}
        \tilde x_{t+1}=\tilde A\tilde x_t+\tilde w_t, \quad
        \psi_t=\tilde C\tilde x_t, \label{eq:enlarged_system}
    \end{align}
    where $\tilde x_t=\begin{bmatrix}
        x_t^\top\ y_t^\top\  \varphi_t^\top\  u_t^\top \  \check x_t^\top
    \end{bmatrix}^\top, \psi_t=\begin{bmatrix}
        y_t^\top \  \check x_t^\top
    \end{bmatrix}^\top$,
    \begin{align}
        \tilde A=\begin{bmatrix}
            A & 0 & 0 & B & 0 \\
            CA & 0 & 0 & CB & 0 \\
            0 & B_{\varphi y} & A_\varphi & 0 & 0 \\
            0 & C_\varphi B_{\varphi y} & C_{\varphi}A_\varphi & 0 & 0 \\
            0 & 0 & 0 & (\boldsymbol{1}_q\otimes I_m) & \check A
        \end{bmatrix},\nonumber
    \end{align}
    $\check x_t$ is defined recursively as:
    \[\check x_{t+1}=\check A\check x_t+(\boldsymbol{1}_{q}\otimes u_t),\quad t=1, \cdots, N,\]
    with $\check x_0=0$, where $\check A=\mathrm{diag}(\mu_1, \cdots, \mu_q)\otimes I_m$, and
    \[\tilde w_t=\begin{bmatrix} B_e & 0 \\ CB_e & 0 \\ 0 & B_{\varphi \epsilon} \\ 0 & C_\varphi B_{\varphi\epsilon} \\ 0 & 0\end{bmatrix}\begin{bmatrix} e_{t+1} \\ \epsilon_t\end{bmatrix}, \tilde C=\begin{bmatrix} 0 & I & 0 & 0 & 0 \\ 0 & 0 & 0 & 0 & I\end{bmatrix}.\]
    It is easy to verify that $\tilde w_t$ is an i.i.d. random noise with zero mean. 
    
    Let \( \mathcal{\check{W}} = \lim_{t \to \infty} \mathbb{E}(\psi_t \psi_t^H) \). We can verify that the augmented system satisfies the assumptions in Lemma~7 of~\cite{liu2020online}, which allows us to prove that for all $\epsilon>0$,
    \[\lim_{N\to\infty}\frac{\frac{1}{N}\sum_{t=1}^{N}\psi_t\psi_t^H-\mathcal{\check W}}{N^{-0.5+\epsilon}}=0\ a.s.\]
    \res{This lemma yields an almost-sure convergence rate of order $N^{-1/2+\epsilon}$
    for the \emph{sample covariance} of outputs of stable LTI systems.
    In the temporally correlated setting considered here, this martingale-based result
    plays a role analogous to that of the law of the iterated logarithm in the i.i.d.\ case.
    }

    Moreover, note that the solution to the least-squares problem~\eqref{eq:finite_least_squares_1}, \( \check{R}_k(N) \), is directly related to the subblocks of \( \frac{1}{N}\sum_{t=1}^N \psi_t \psi_t^H \), provided that \( N \) is sufficiently large to ensure that the solution 
    \[
    [\check{R}_1(N)\ \cdots\ \check{R}_q(N)] = \left(\frac{1}{N} \sum_{t=1}^N y_t \check{x}_t^H\right) \left(\frac{1}{N} \sum_{t=1}^N \check{x}_t \check{x}_t^H\right)^{-1},
    \]
    exists. Meanwhile, the asymptotic solution \( \check{R}_k^* \) exhibits the same relationship with the subblocks of \( \mathcal{\check{W}}\). To explicitly express this relationship, we define the matrix function \( \mathcal{A}(X) \) for a given matrix \( X\in\mathbb C^{(p+qm)\times (p+qm)} \) as follows:

    \begin{align}
        &\mathcal A(X)\triangleq\begin{bmatrix}
        I_p & 0
    \end{bmatrix}(\mathcal{\check W}-X)\begin{bmatrix}
        0 \\ I_{qm}
    \end{bmatrix}\nonumber \\
    &\qquad\qquad\qquad\left(\begin{bmatrix}
        0 & I_{qm}\end{bmatrix}
        (\mathcal{\check W}-X)\begin{bmatrix}
        0 \\ I_{qm}
        \end{bmatrix}\right)^{-1}.
    \end{align}
    By Assumption~\ref{assump:input}, $\lim_{t\to\infty}\mathbb{E}(\check x_t\check x_t^H)$ is invertible and for sufficiently large $N$, $\frac{1}{N}\sum_{t=1}^N \check x_t\check x_t^H$ is invertible. Therefore, the following equation is well-defined and holds:
    \begin{equation*}
        \begin{aligned}
            &\begin{bmatrix}\check R_1(N) & \cdots & \check R_q(N)\end{bmatrix}=\mathcal{A}\left(\mathcal{\check W}-\frac{1}{N}\sum_{t=1}^N\psi_t\psi_t^H\right), \\
            & \begin{bmatrix}\check R_1^* & \cdots & \check R_q^*\end{bmatrix}=\mathcal{A}(0).
        \end{aligned}
    \end{equation*}
    Taking notice of the fact that $\mathcal{A}(X)$ is differentiable at $0$, by Lemma~3~3) of~\cite{liu2020online}\res{, which ensures that the convergence of
    $\frac{1}{N}\sum_{t=1}^N \psi_t\psi_t^H$ (and its rate) is preserved under the differentiable mapping $\mathcal{A}(\cdot)$, we obtain that} for all $\epsilon>0$,
    \begin{equation}\label{eq:R_rate}
        \lim_{N\to\infty}\frac{\|\check R_k(N)-\check R_k^*\|}{N^{-0.5+\epsilon}}=0\ a.s,\quad k=1, 2, \cdots, q.
    \end{equation}
    Finally, by the definition of the $\mathcal H_2$ norm,
    \begin{align}
        &\|\check G_N(z)-\check G^*(z)\|_2\nonumber \\
        &=(\mathrm{tr}[(\check R(N)-\check R^*)\Xi_{\boldsymbol{\mu, \boldsymbol{\mu}}}(\check R(N)-\check R^*)^H])^{1/2} \nonumber \\
        &\leq\|\Xi_{\boldsymbol{\mu}, \boldsymbol{\mu}}\|^{1/2}\|\check R(N)-\check R^*\|,\label{eq:yopt_ycheck_bound}
    \end{align}
    where $\|\cdot\|$ denotes the Frobenius norm of a matrix, $\check R(N)=[\check R_1(N)\ \cdots\ \check R_q(N)], \check R^*=[\check R_1^*\ \cdots\ \check R_q^*]$ and 
    \begin{align}
        \Xi_{\boldsymbol{\mu}, \boldsymbol{\mu}}=\begin{bmatrix}
            \frac{1}{1-\mu_1\bar\mu_1} & \frac{1}{1-\mu_1\bar\mu_2} & \cdots & \frac{1}{1-\mu_1\bar\mu_q} \\
            \frac{1}{1-\mu_2\bar\mu_1} & \frac{1}{1-\mu_2\bar\mu_2} & \cdots & \frac{1}{1-\mu_2\bar\mu_q} \\
            \vdots & \vdots & \ddots & \vdots \\
            \frac{1}{1-\mu_q\bar\mu_1} & \frac{1}{1-\mu_q\bar\mu_2} & \cdots & \frac{1}{1-\mu_q\bar\mu_q}
        \end{bmatrix}\otimes I_m.\label{eq:xi_mu_mu}
    \end{align}
    Moreover, $\|\Xi_{\boldsymbol{\mu}, \boldsymbol{\mu}}\|^{1/2}\leq \frac{q\sqrt{m}}{\sqrt{1-\rho_\mu^2}}$, where $\rho_\mu=\max_{k=1, \cdots, q}|\mu_k|$.

    Hence, the theorem is proved by combining~\eqref{eq:R_rate} with~\eqref{eq:yopt_ycheck_bound}.
\end{proof}

\subsection{Asymptotic Bias Analysis}
This subsection quantifies the bias between $G(z)$ and the asymptotic result $\check G^*(z)$ of the least-squares algorithm~\eqref{eq:finite_least_squares_1}. For this purpose, we introduce the following additional assumptions:
\begin{assumption}\label{assump:bias}
    \begin{itemize}
        \item The system $G(z)$ has no repeated poles.
        \item The poles of the system $G(z)$ are within a closed continuum $\mathcal{D}\subset\mathbb{D}$, where $\mathcal{D}$ is known.
    \end{itemize}
\end{assumption}
\begin{remark}
    The assumption of no repeated poles is standard in the OBF literature~\cite{ninness_asymptotic_1996, van_den_hof_system_2005}. Furthermore, the assumption that the pole region $\mathcal D$ is known based on \emph{a priori} information is common in studies focusing on the pole selection algorithms for OBFs~\cite{van_den_hof_system_2005,toth_asymptotically_2009}.
\end{remark}
As a result, the transfer function $G(z)$ of the system can be written in the following form~\cite[Chapter 11]{van_den_hof_system_2005}:
    \begin{equation}\label{eq:tf_decomposition}
        G(z)=\sum_{j=1}^n\frac{R_j}{z-\lambda_j}, R_j\in\mathbb{C}^{p\times m},
    \end{equation}
where $\lambda_j$ are the \emph{distinct} poles of the system, and $n$ denotes the number of poles. 

Let $\Phi_u(\omega)$ denote the power spectral density (PSD) of the input signal $u_t$. Specifically,
\[\Phi_u(\omega)=\sum_{k=-\infty}^\infty R_u(k)e^{ik\omega},\]
where $R_u(k)=\mathbb{\bar E}(u_tu_{t-k}^H)$ denotes the autocorrelation function of $u_t$.

The following theorem provides an upper bound on the asymptotic bias of the least-squares algorithm~\eqref{eq:finite_least_squares_1} using the bases $V_k(z)$ as the number of samples tends to infinity. The proof of the theorem is provided in Appendix~\ref{append:system_approx_error}.
\begin{theorem}[Asymptotic bias of the least-squares algorithm]\label{cor:system_approx_error}
    For a fixed group of bases $\{1/(z-\mu_k)\}_{k=1}^q$, the system approximation bias of $G$ using the asymptotic solution of the least-squares problem $\check G^*(z)=\sum_{k=1}^q \frac{\check R_k^*}{z-\mu_k}$ satisfies the inequality
    \begin{align}
        &\|\check G^*(z)-G(z)\|_2\nonumber \\
        &\leq \left[1+\frac{\mathrm{ess}\sup_{\omega}\|\Phi_u(\omega)\|}{\mathrm{ess}\inf_{\omega}\|\Phi_u(\omega)\|}\right]\sum_{j=1}^n \frac{\|R_j\|}{\sqrt{1-|\lambda_j|^2}}\prod_{k=1}^q [\lambda_j, \mu_k]_{\mathrm h} \nonumber\\
        &\leq \left[1+\frac{\mathrm{ess}\sup_{\omega}\|\Phi_u(\omega)\|}{\mathrm{ess}\inf_{\omega}\|\Phi_u(\omega)\|}\right]\frac{\bar R}{\sqrt{1-\rho_\lambda^2}}\max_{j=1, \cdots, n}\prod_{k=1}^q[\lambda_j, \mu_k]_{\mathrm{h}}, \label{eq:obf_bias}
    \end{align}
    where $\bar R=\sum_{j=1}^n \|R_j\|$ and $\rho_\lambda=\max_{j=1, \cdots, n}|\lambda_j|$.
\end{theorem}
\begin{remark}
    The $\mathcal H_2$ error bounds provided herein are consistent with the worst-case $\mathcal{H}_2$ identification and model reduction problems (see, e.g., \cite{Giarre97, Antoulas08}), and resonate with mean-square robustness analyzes in LTI systems (see, e.g., \cite{Qi2017a, Qi2017b, Bamieh2020, Jianqi2021}), and hence are compatible with system analysis and synthesis in mean-square designs. 
\end{remark}
The upper bound in Theorem~\ref{cor:system_approx_error} can be decomposed into two components. The first component is determined by the intrinsic properties of the system ($\bar{R}, \rho_{\lambda}$) and the inputs in the collected data ($\Phi_u(\omega)$), while the second component depends on $\tau(\boldsymbol{\lambda}, \boldsymbol{\mu}) = \max_{j=1, \dots, n} \prod_{k=1}^q [\lambda_j, \mu_k]_{\mathrm{h}}$, which is influenced by the choice of $\mu_k$. Motivated by this observation, in the following section, we propose a pole selection algorithm that minimizes $\tau(\boldsymbol{\lambda}, \boldsymbol{\mu})$, thereby reducing the asymptotic bias.

\section{Fundamental Lower Bound and Tsuji Pole Selection Algorithm}\label{sec:pole_allocation_algorithm}
This section focuses on \emph{robustly}\footnote{One possible approach is to \emph{adaptively} select $\mu_k$ to be as close as possible to the true system poles $\lambda_j$ based on the sampled data. However, in system identification, the values of $\lambda_j$ are generally unknown. In Section~\ref{sec:fundamental_limit}, we show that identifying the true system poles becomes exponentially difficult as the system dimension increases.} selecting the poles $\mu_k$ by minimizing the worst-case identification bias across a class of systems, specifically, those whose poles lie within the pole region\footnote{The pole region $\mathcal{D}$ is a complex subset of the open unit disk that contains all true poles~\cite[Chapter 11]{van_den_hof_system_2005}.} $\mathcal{D}$, determined by \emph{a priori} information. Additionally, we establish a fundamental lower bound on the worst-case identification bias and show that the Tsuji points asymptotically achieve this bound as the number of \res{OBF poles $q$} tends to infinity, demonstrating the optimality of the Tsuji pole selection algorithm.

The following minimax problem based on the pole region $\mathcal D$ can be used to robustly select the optimal poles $\mu_k$:
\begin{problem}[Minimax pole selection problem~\res{\cite[Chapter 11]{van_den_hof_system_2005}}]\label{prob:minimax_problem}
    \begin{equation}\label{eq:minimax_problem}
            \min_{\mu_1, \cdots, \mu_q\in\mathcal D}\max_{\lambda\in\partial\mathcal D}\prod_{k=1}^q [\lambda, \mu_k]_{\mathrm{h}},
    \end{equation}
    where $\partial \mathcal D$ denotes the boundary of $\mathcal D$.
\end{problem}
\res{Here the maximization is restricted to $\partial\mathcal D$ by the maximum-modulus principle,
applied to the function $\lambda \mapsto \prod_{k=1}^q[\lambda,\mu_k]_{\mathrm h}$.}

However, for a general pole region $\mathcal D$, this minimax problem is equivalent to a semi-infinite programming problem~\cite{semi_infinite_programming}, and is computationally challenging to solve. Instead, we propose to select the poles by solving the following maximization problem:
\begin{problem}[Tsuji pole selection problem]\label{prob:max_opt_problem}
\begin{equation}\label{eq:solve_tsuji_points_parametrize}
    \begin{aligned}
        [\eta_{q1}, \cdots, \eta_{qq}]=\arg\max_{\mu_1, \cdots, \mu_q\in\partial \mathcal{D}}\ \prod_{1\leq k<l\leq q}  [\mu_k, \mu_l]_{\mathrm{h}},
    \end{aligned}
\end{equation}
where $\partial\mathcal{D}$ denotes the boundary of $\mathcal{D}$.
\end{problem}
The solutions $\eta_{q1}, \cdots, \eta_{qq}$ are called the $q$-th Tsuji points of the region $\mathcal D$ introduced in Definition~\ref{def:transfinite_diameter_tsuji_point}. The following theorem demonstrates the performance of the Tsuji points:
\begin{theorem}\label{thm:tsuji_asymptotic_optimal}
    Let $\eta_{q1}, \cdots, \eta_{qq}$ denote the $q$-th Tsuji points of $\mathcal{D}$. Then, by choosing $\mu_k=\eta_{qk}$ for Problem~\ref{prob:minimax_problem}, the exponential decay rate of its objective function asymptotically approaches the hyperbolic Chebyshev constant $\tau(\mathcal{D})$ as $q\to\infty$, i.e.,
    \begin{equation}
        \lim_{q\to\infty}\max_{z\in\mathcal{D}} \left(\prod_{k=1}^q [z, \eta_{qk}]_{\mathrm{h}}\right)^{1/q}=\tau(\mathcal{D}).
    \end{equation}
\end{theorem}
\begin{proof}
    Suppose $\eta_{q1}, \cdots, \eta_{qq}$ are the $q$-th Tsuji points of the set $\mathcal{D}$, i.e.,
    \begin{equation*}
        \max_{\mu_1, \cdots, \mu_q\in\mathcal{D}}\prod_{1\leq k<l\leq q}[\mu_k, \mu_l]_{\mathrm{h}}=\prod_{1\leq k<l\leq q}[\eta_{qk}, \eta_{ql}]_{\mathrm{h}}.
    \end{equation*}
    Then, for each $l=1, \cdots, q$, one can verify that
    \[\phi_{-l}\triangleq\max_{z\in\mathcal{D}}\prod_{k=1, \cdots, q, k\neq l}[z, \eta_{qk}]_{\mathrm{h}}=\prod_{k=1, \cdots, q, k\neq l}[\eta_{ql}, \eta_{qk}]_{\mathrm{h}}.\]
    Multiplying all $\phi_{-l}$ together, we can get that
    \begin{equation*}
        \prod_{l=1}^q \phi_{-l}=\prod_{1\leq k<l\leq q}\left([\eta_{qk}, \eta_{ql}]_{\mathrm{h}}\right)^2.
    \end{equation*}
    Therefore, there exists an $\ell$ with $1\leq \ell\leq q$, such that
    \begin{equation*}
        \phi_{-\ell}\leq \left(\prod_{l=1}^q \phi_{-l}\right)^{1/q}=\prod_{1\leq k<l\leq q}([\eta_{qk}, \eta_{ql}]_{\mathrm{h}})^{2/q}.
    \end{equation*}
    Consequently,
    \begin{align}
        \max_{z\in\mathcal{D}}\left(\prod_{k=1}^q [z, \eta_{qk}]_{\mathrm{h}}\right)^{1/q}&< \left(\phi_{-\ell}\right)^{1/q}\nonumber \\
        &\leq \prod_{1\leq k<l\leq q}([\eta_{qk}, \eta_{ql}]_{\mathrm{h}})^{2/q^2},\nonumber
    \end{align}
    where the first inequality is because $[z, \eta_{q\ell}]_{\mathrm{h}}<1$ for all $z, \eta_{q\ell}\in\mathbb{D}$. Let $q\to\infty$, we have that
    \begin{equation}
        \begin{aligned}
            &\limsup_{q\to\infty}\max_{z\in\mathcal{D}}\left(\prod_{k=1}^q [z, \eta_{qk}]_{\mathrm{h}}\right)^{1/q} \\
            &\leq \limsup_{q\to\infty}\max_{\mu_1, \cdots, \mu_q\in\mathcal{D}}\left[\prod_{1\leq k<l\leq q}[\mu_k, \mu_l]_{\mathrm{h}}^{2/(q(q-1))}\right]^{1-1/q} \\
            &=\tau(\mathcal{D})=\lim_{q\to\infty}\min_{\mu_1, \cdots, \mu_q}\max_{z\in\mathcal{D}}\left(\prod_{k=1}^q [z, \mu_k]_{\mathrm{h}}\right)^{1/q} \\
            &\leq \liminf_{q\to\infty}\max_{z\in\mathcal{D}}\left(\prod_{k=1}^q [z, \eta_{qk}]_{\mathrm{h}}\right)^{1/q}.\nonumber
        \end{aligned}
    \end{equation}
    Hence, we conclude that
    \begin{equation}
        \lim_{q\to\infty}\max_{z\in\mathcal{D}}\left(\prod_{k=1}^q [z, \eta_{qk}]_{\mathrm{h}}\right)^{1/q}=\tau(\mathcal{D}).\nonumber
    \end{equation}
\end{proof}

Finally, we leverage the results above to evaluate the worst-case identification bias for a group of systems consistent with \emph{a priori} information, using the same set of bases determined by the proposed pole selection algorithm. Let $\mathscr G$ denote this group of target systems whose poles lie within the region $\mathcal D$ and whose \res{coefficient magnitude} satisfies $\sum_{j=1}^n\|R_j\|\leq \bar R$. Since we expect an exponential decay as $q$ increases, we focus on the following term:

\begin{equation}\label{eq:fundamental_limit_pole}
	\liminf_{q\rightarrow \infty}\max_{G\in\mathscr{G}}\min_{\check R_1, \cdots, \check R_q}\left\|G-\sum_{k=1}^q \frac{\check R_{k}}{z-\mu_{k}}\right\|_2^{1/q}.
\end{equation}
\begin{theorem}[Fundamental limit on the worst-case approximation bias]\label{thm:system_approximation_error}
	For any sequence of the selected poles $\{\mu_{k}\}$,
    \begin{equation}\label{eq:system_approximation_error_lim}
			\liminf_{q\to\infty}\max_{G\in\mathscr{G}}\min_{\check R_1, \cdots, \check R_q}\left\|G-\sum_{k=1}^q \frac{\check R_{k}}{z-\mu_{k}}\right\|_2^{1/q}\geq \tau(\mathcal D),
   \end{equation}
    where $\tau(\cdot)$ is the hyperbolic Chebyshev constant introduced in Definition~\ref{def:chebyshev_constant}. 
    
    On the other hand, when $\{\mu_k\}$ are chosen as the $q$-th Tsuji points $\eta_{q1}, \cdots, \eta_{qq}$ of $\mathcal D$, the following equality holds:
    \begin{align}
        \lim_{q\to\infty}\max_{G\in\mathscr G}\min_{\check R_1, \cdots, \check R_q}\left\|G-\sum_{k=1}^q \frac{\check R_{k}}{z-\eta_{qk}}\right\|_2^{1/q}=\tau(\mathcal D).
    \end{align}
\end{theorem}
The proof of the theorem is reported in Appendix~\ref{append:system_approx_error}.

Theorem~\ref{thm:system_approximation_error} establishes a fundamental limit, proving that \emph{no} pole selection algorithm can achieve a worst-case approximation bias with a decay rate faster than $O((\tau(\mathcal D)-\epsilon)^q)$ for any $\epsilon>0$. 

On the other hand, the theorem also proves that the Tsuji points, which are the solution to the proposed maximization problem~\eqref{eq:solve_tsuji_points_parametrize}, asymptotically achieve this limit. Consequently, our Tsuji pole selection approach is asymptotically optimal while avoiding the computationally expensive minimax formulation.

\subsection{Initialization Strategy}
In the preceding discussion, we have simplified a minimax problem to a maximization problem, whose solution is asymptotically optimal. However, Problem~\ref{prob:max_opt_problem} remains non-convex in general, and thus the initial choice of $\mu_k$ may affect the optimization result. To effectively solve Problem~\ref{prob:max_opt_problem}, we propose an \emph{initialization strategy} for the optimization problem. 

Let $f$ denote a conformal mapping that transforms the annulus $\mathbb{D}\backslash\mathbb{\bar D}_{\tau(\mathcal{D})}=\{z\mid \tau(\mathcal{D})<|z|<1\}$ to $\mathbb{D}\backslash\mathcal{D}$, such that $|z|=\tau(\mathcal{D})$ corresponds to $\partial\mathcal{D}$. The existence and computation of this mapping can be found in~\cite{trefethen_numerical_2020}. When $\mathcal{D}$ is a real interval $[-\rho, \rho], \rho<1$, the analytical solution of its conformal mapping is provided in Appendix~\ref{append:tsuji_points}. 

The proposed \emph{initialization strategy} to solve Problem~\ref{prob:max_opt_problem} is summarized as follows:
\begin{enumerate}
    \item \textbf{Find the conformal mapping} on the complex plane $f:\mathbb{D}\to\mathbb{D}$ that maps the annulus $\mathbb{D}\backslash\mathbb{\bar D}_{\tau(\mathcal{D})}$ to the open set $\mathbb{D}\backslash\mathcal{D}$.
    \item \textbf{Uniformly sample $q$ points} $\tau(\mathcal{D})e^{2k\pi i/q}, k=1, \cdots, q$ on the inner circle $\mathbb{D}_{\tau(\mathcal{D})}=\{z\in\mathbb{C}\mid |z|=\tau(\mathcal{D})\}$.
    \item \textbf{Map the sampled points} to the pole region $\mathcal{D}$ via the conformal mapping $f(\cdot)$, i.e., $\tilde\eta_{qk}=f(\tau(\mathcal{D})e^{2k\pi i/q}), k=1, \cdots, q$.
\end{enumerate}
The resulting $\tilde \eta_{qk}$ are the initializations to the Tsuji points. 

The initialization procedure when $\mathcal D$ is an ellipse inside the unit circle is visualized in Fig.~\ref{fig:tsuji_complex}. The rationale behind this strategy is that $\tilde\eta_{q1}, \cdots, \tilde\eta_{qq}$ asymptotically converge to the Tsuji points when $\partial\mathcal{D}$ is an analytic Jordan curve, as shown by the following result:

\begin{figure}
    \centering
    \includegraphics[width=0.45\textwidth]{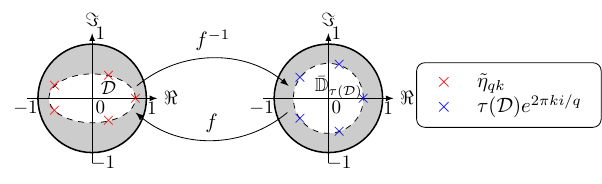}
    \caption{Visualization of the introduced initialization strategy of the Tsuji points given the pole region $\mathcal D$. $f$ in the figure denotes the conformal mapping from the annulus $\mathbb D\backslash \mathbb{\bar D}_{\tau(\mathcal D)}$ to $\mathbb D\backslash\mathcal D$ (the gray regions).}
    \label{fig:tsuji_complex}
\end{figure}

\begin{proposition}[Distribution of Tsuji points; \res{see Theorem~2 in~\cite{menke_distribution_1985} and equation~(3) in}~\cite{stiemer_approximation_2005}]\label{thm:distribution_of_tsuji_points}
    Denote $\eta_{q1}, \cdots, \eta_{qq}$ as the $q$-th Tsuji points of $\mathcal{D}$, and let $\nu_k$ be points such that $f(\tau(\mathcal{D})e^{i\nu_{k}})=\eta_{qk}, k=1, \cdots, q$, then
    \begin{equation}
        \left|\nu_{k}-\frac{2\pi k}{q}\right|\leq L\frac{(\log q)^{3/2}}{q}, k=1, \cdots, q,
    \end{equation}
    where $L>0$ is independent of $q$ and $k$.
\end{proposition}

The proposed pole selection algorithm with the initialization strategy is summarized as follows. Suppose we have found the explicit form of the conformal mapping $f$.
\begin{enumerate}
    \item \textbf{Initialize} $\nu_k=2k\pi /q$.
    \item \textbf{Solve the following maximization problem}:
    \begin{equation}
        \max_{\nu_1, \cdots, \nu_q\in[0, 2\pi]}\prod_{1\leq k<l\leq q}[f(\tau(\mathcal D)e^{i\nu_k}), f(\tau(\mathcal D)e^{i\nu_l})]_{\mathrm h},
    \end{equation}
    where the optimal solutions are denoted as $\nu_1^*, \cdots, \nu_q^*$.
    \item The \textbf{Tsuji points} are $\eta_{qk}=f(\tau(\mathcal D)e^{i\nu_k^*})$.
\end{enumerate}

\section{Hardness of Precise Pole Identification}\label{sec:fundamental_limit}
The previous section proposes a \emph{robust} pole selection algorithm by minimizing the worst-case identification bias across a group of systems. However, for identifying a single system, a more direct approach is to \emph{adaptively} select the poles $\mu_k$ to approximate the true system poles $\lambda_j$ as closely as possible based on sampled data. 

In this section, we show that for an LTI system whose transfer function has a denominator of degree $n$, the number of samples required to distinguish the true system from the \emph{robust solution} in the previous section grows exponentially with $n$. Furthermore, we extend this result to a state-space formulation, proving that the sample complexity for distinguishing the true state-space model of order $n$ from the \emph{robust solution} of order $n$ grows exponentially with $\underline{n} \triangleq \left\lfloor \frac{n}{\min(p, m)} \right\rfloor$. Consequently, identifying the true system poles $\lambda_j$ becomes exponentially challenging as the system's dimension increases.

Specifically, suppose the true system has the dynamics~\eqref{eq:true_system}. Besides Assumption~\ref{assump:system_stability},~\ref{assump:input} and~\ref{assump:bias}, for technical simplicity, we additionally impose 
\begin{assumption}\label{assump:fundamental_limit}
    \begin{itemize}
        \item The noise $\{v_t\}$ is i.i.d. and Gaussian distributed, i.e., $v_t\sim\mathcal{N}(0, \mathcal{R})$ with $\mathcal{R}\succ 0$.
        \item The samples are collected starting from time $0$, and the system inputs from time $-\infty$ to $-1$ are kept at $0$.
        \item The system input is a stationary process and its power spectral density is bounded, i.e., $\mathrm{ess}\sup_{\omega}\|\Phi_u(\omega)\|<\infty$.
    \end{itemize}
\end{assumption}

Then, the transfer function $G(z)$ can be similarly decomposed as in~\eqref{eq:tf_decomposition}. We next use the robust solution in the previous section to construct a surrogate system $\check G$ as follows:
\begin{itemize}
    \item The poles $\mu_1, \cdots, \mu_{n}$ are fixed as the minimizer of Problem~\ref{prob:minimax_problem} in Section~\ref{sec:pole_allocation_algorithm};
    \item The parameters $\check R_1, \cdots, \check R_{n}$ are chosen as the optimal solution of Problem~\ref{prob:system_identification}.
\end{itemize}
Here, we keep the dimension of the surrogate system the same as the true system for the fairness of comparison. The constructed surrogate system $\check G$ can be written as
\begin{equation}\label{eq:fundamental_limit_surrogate_system}
    \check G(z)=\sum_{k=1}^{n} \frac{\check R_k}{z-\mu_k}.
\end{equation}
In what follows, we establish the aforementioned sample complexity bound via distinguishing the distributions of $y_1, \cdots, y_N, u_0, \cdots, u_{N-1}$ under the following two hypotheses:
\begin{equation}\label{eq:hypothesis}
    \begin{aligned}
        &\mathcal{H}_0: \{y_t, u_{t-1}\}_{t=1}^N\text{ are from the true system }G, \\
        &\mathcal{H}_1:\{y_t, u_{t-1}\}_{t=1}^N\text{ are from the surrogate system }\check G.
    \end{aligned}
\end{equation}
Let $\mathbb{P}_0$ and $\mathbb{P}_1$ denote the conditional distribution of $\{y_t\}_{t=1}^N$ given $\{u_t\}_{t=0}^{N-1}$ in the two hypotheses $\mathcal{H}_0$ and $\mathcal{H}_1$ respectively. Here, we use the KL divergence between the distributions of the system outputs $y_t$ given the same inputs $u_t$ as a metric on the difficulty of distinguish the true system $G$ from the surrogate system $\check G$. The following theorem shows that the sample complexity for distinguishing the two systems, i.e., reach a constant KL divergence requirement, grows exponentially with $n$. The proof of the theorem is reported in Appendix~\ref{append:fundamental_limit}.
\begin{theorem}\label{thm:fundamental_limit}
    The expected\footnote{The expectation is taken with respect to the persistent exciting input $u_t$.} KL divergence between the two distributions in~\eqref{eq:hypothesis} satisfies
    \begin{equation}
       \mathbb E D_{\mathrm{KL}}(\mathbb{P}_0\|\mathbb{P}_1)\leq \frac{\bar R^2\|\mathcal{R}^{-1}\|\mathrm{ess}\sup_{\omega}\|\Phi_u(\omega)\|}{2(1-\rho_{\lambda}^2)}\tau_n(\mathcal{D})^{2n}N,
    \end{equation}
    where $\rho_\lambda$ is the spectral radius of the true system $G$, $\mathcal{R}$ is the covariance matrix of the noise in Assumption~\ref{assump:fundamental_limit}, $\bar R=\sum_{j=1}^n \|R_j\|$ is the modified system energy with slight abuse of notations and $\tau_n(\mathcal{D})$ is the finite hyperbolic Chebyshev constant of the region $\mathcal{D}$ defined in Definition~\ref{def:chebyshev_constant}.

    Moreover, given a constant $\delta>0$, in order to distinguish the two hypotheses in~\eqref{eq:hypothesis} with the KL divergence no smaller than $\delta$, $N$ satisfies
    \[N\geq \frac{2\delta(1-\rho_\lambda^2)}{\bar R^2\|\mathcal{R}^{-1}\|\mathrm{ess}\sup_{\omega}\|\Phi_u(\omega)\|}\tau_n(\mathcal{D})^{-2n}.\]
    When the system dimension $n$ tends to infinity,
    \[\liminf_{n\to\infty}N^{1/(2n)}\geq \tau(\mathcal D)^{-1}.\]
\end{theorem}
Theorem~\ref{thm:fundamental_limit} shows that the number of samples required to distinguish the true system from the surrogate system grows exponentially with the system dimension $n$, even though the two systems have entirely distinct poles. This result implies that identifying the true system poles becomes exponentially more challenging as the system dimension increases.

\begin{remark}
    Note that $\tau_n(\mathcal{D})$ converges to the hyperbolic Chebyshev constant of $\mathcal{D}$, which is less than $1$ for a strictly stable system. Consequently, the sample complexity grows exponentially with $n$. Moreover, as is illustrated by Lemma~\ref{lemma:hyperbolic_chebyshev_constant}, if the region of poles $\mathcal{D}$ is a real interval, $\tau(\mathcal{D})$ is usually small, and the sample complexity grows exponentially with $n$ at a fast speed. For example, even when $\mathcal{D}=[-0.999, 0.999]$, ${N\sim \mathcal{O}(1.81^{n})}$.
\end{remark}
\begin{remark}
    The result can be generalized to unstable systems by establishing the hardness on the identification of the stable subsystem. To be specific, by partial fraction decomposition, the transfer function matrix $G(z)$ can be rewritten into
    \[G(z)=G_s(z)+G_n(z),\]
    where $G_s(z)$ contains all the stable poles and $G_n(z)$ contains all the unstable or marginally stable poles. Even if $G_n(z)$ is perfectly known from an oracle, which simplifies the identification problem, we can still consider the modified system output
    \[\tilde y_t=y_t-G_n(z)u_t = G_s(z)u_t+v_t,\]
    and the sample complexity of the identification  has a lower bound defined by the sample complexity of the stable subsystem.
\end{remark}

Additionally, motivated by recent studies highlighting the ill-conditioned nature of \emph{state-space model identification}~\cite{tsiamis_linear_2021,li2022fundamental}, we extend our results to state-space models. The proof of this corollary appears in Appendix~\ref{append:fundamental_limit}.
\begin{corollary}\label{cor:fundamental_limit_ss}
    Recall that $\underline n=\lfloor\frac{n}{\min(p, m)}\rfloor$. For any $n$-dimensional state-space model with $m$ inputs, $p$ outputs and parameters $A, B, C$, whose corresponding transfer function $G$ and the observation noises satisfy Assumption~\ref{assump:system_stability},~\ref{assump:bias} and Assumption~\ref{assump:fundamental_limit}, there exists a corresponding $n$-dimensional state-space realization of the constructed $\check G$, such that the sample complexity for differentiating the two hypotheses in~\eqref{eq:hypothesis} with the KL divergence no smaller than $\delta$ satisfies
    \begin{equation}
        \liminf_{n\to\infty} N^{1/(2\underline n)}\geq \tau(\mathcal D)^{-1}.
    \end{equation}
\end{corollary}
\begin{rese}
\begin{remark}
    Theorem~\ref{thm:fundamental_limit} and Corollary~\ref{cor:fundamental_limit_ss} indicate that any algorithm attempting to estimate the true pole locations from finite noisy data, including both classical state-space identification methods such as Ho-Kalman and MOESP and more recent approaches that aim to recover state-space parameters from input-output trajectories~\cite{oymak_revisiting_2022,Tsiamis2019}, may be intrinsically ill-conditioned, particularly as the system dimension increases.
\end{remark}
\end{rese}

\section{Simulations}\label{sec:simulation}
This section provides numerical examples to verify the derived bounds and illustrate the effectiveness of the proposed pole selection algorithm.

We begin by testing the performance of the pole selection algorithm in identifying the system
\begin{equation}\label{eq:small_system_equation}
    G(z)=\frac{0.0247z+0.0355}{\prod_{j=1}^4(z-\lambda_j)},
\end{equation}
using $q$ OBFs, where
\[\lambda_1\sim\mathcal N(0, 0.02^2), \lambda_2\sim\mathcal N(0, 0.02^2),\]
\[\lambda_3\sim\mathcal N(0.9048, 0.02^2), \lambda_4\sim\mathcal N(0.3679, 0.02^2).\]
Moreover, all poles $\lambda_j$ are projected back to the interval $[-0.95, 0.95]$, which is used as \emph{a priori} information. The expected location of the true poles and the numerator of the transfer function in~\eqref{eq:small_system_equation} are chosen according to the system used in~\cite{mi_frequency-domain_2012, ninness_unifying_1997_cdc}. We further add noise to the true poles to test the robustness of the pole selection algorithms. A total of $100$ independent experiments are conducted.

We apply the following pole selection algorithms to the perturbed systems:
\begin{itemize}
    \item \emph{Minimax method~\cite{van_den_hof_system_2005}}: The minimax optimization problem (Problem~\ref{prob:minimax_problem}) is solved using MATLAB's minimax solver \res{\texttt{fminimax}}, assuming the pole region is $\mathcal D=[-0.95, 0.95]$.
    \res{\item \emph{SQP method~\cite{bachnas2023advancing}}:
    The min-max pole selection problem (Problem~\ref{prob:minimax_problem}) is solved using a multi-start Sequential Quadratic Programming (SQP) approach implemented via MATLAB's \texttt{fmincon}, following~\cite[Chapter 3.4.1, Algorithm 1]{bachnas2023advancing}, with pole region $\mathcal D=[-0.95, 0.95]$.}
    \res{\item \emph{RA method~\cite{bachnas2023advancing}}:
    The min-max pole selection problem (Problem~\ref{prob:minimax_problem}) is solved using the Randomized Algorithm (RA) introduced in~\cite[Chapter 3.4.2, Algorithm 3]{bachnas2023advancing}, based on probabilistic performance verification and bisection search, with pole region $\mathcal D=[-0.95, 0.95]$.}
    \item \emph{Greedy method~\cite{mi_frequency-domain_2012}}: \res{This method selects poles sequentially in an adaptive manner. At each 
    iteration, it computes the \emph{residual function}, defined as the part 
    of the true system's transfer function that is not yet captured by the 
    previously selected OBF bases. The algorithm then chooses the pole whose 
    associated basis function is most aligned with this residual in the 
    $\mathcal{H}_2$ sense, updates the residual accordingly, and repeats. 
    In our implementation, the true system transfer function is used directly 
    as the nominal system~\eqref{eq:small_system_equation}.}
    \item \emph{Initial guess of Tsuji points}: The initial Tsuji points are calculated using the method in Section~\ref{sec:pole_allocation_algorithm} and Appendix~\ref{append:tsuji_points}, with pole region $\mathcal D=[-0.95, 0.95]$. 
    \item \emph{Tsuji points}: The Tsuji points are solved by the maximization Problem~\ref{prob:max_opt_problem} using the previously calculated initial guess.
\end{itemize}
Performance of identification is measured using the relative asymptotic $\mathcal H_2$ bias, defined as the ratio between the optimal value of Problem~\ref{prob:system_identification} and the $\mathcal H_2$ norm of the true system $\|G(z)\|_2$. 

Fig.~\ref{fig:small_system_bias} shows the relative $\mathcal H_2$ approximation bias versus the number of poles $q$. The figure demonstrates that the asymptotic approximation bias of our Tsuji pole selection algorithm decreases exponentially with the number of basis functions, confirming the theoretical results from Section~\ref{sec:main_results}.

\res{To further interpret the observed trends in Fig.~\ref{fig:small_system_bias},
we compare the behavior of different pole selection methods.
Several numerical solvers, including minimax optimization, SQP, and RA,
are designed to approximate the global solution of the minimax pole selection problem
characterized by Problem~\ref{prob:minimax_problem}. It can be observed that for a moderate number of bases, these optimization methods are more likely to attain
the global optimum of the minimax problem, and when this occurs, the resulting
approximation errors are very small.
This behavior is evident for $q=10$, where the minimax-based solution achieves the smallest approximation bias among all methods.} \res{As the number of bases increases, however, the optimization problem becomes more nonconvex and sensitive to initialization, leading to degraded performance and increased variability across Monte Carlo trials. For $q=15$, the minimax-based solvers exhibit noticeably larger worst-case errors. In contrast, the Tsuji points remain stable and consistently achieve small \emph{worst-case} asymptotic biases.}

\res{From a robustness perspective, the worst-case approximation bias is a more meaningful metric than best-case performance, since all methods may occasionally achieve small errors when the selected basis poles happen to be well aligned with the dominant poles of the true system, especially for low-dimensional systems. In this sense, the Tsuji-based methods consistently exhibit small and stable
\emph{worst-case} asymptotic biases, which aligns well with
their theoretical design objective of minimizing the worst-case approximation error.}

\begin{figure}[!htbp]
    \centering
    \begin{subfigure}{0.45\textwidth}
        \centering
        \includegraphics[width=\textwidth]{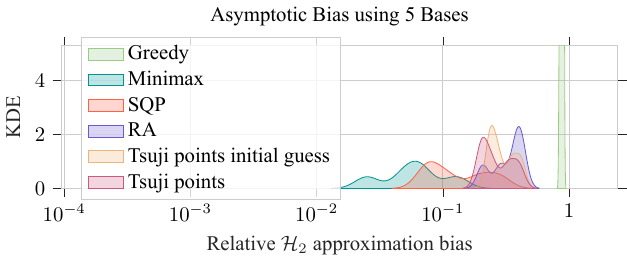}
    \end{subfigure}
    \begin{subfigure}{0.45\textwidth}
        \centering
        \includegraphics[width=\textwidth]{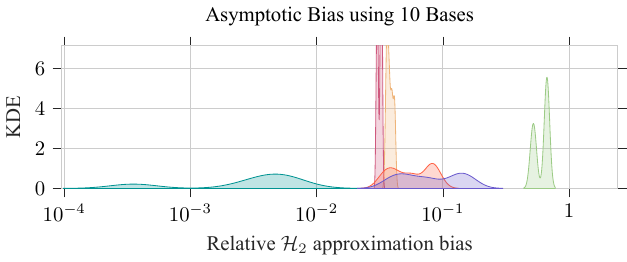}
    \end{subfigure}
    \begin{subfigure}{0.45\textwidth}
        \centering
        \includegraphics[width=\textwidth]{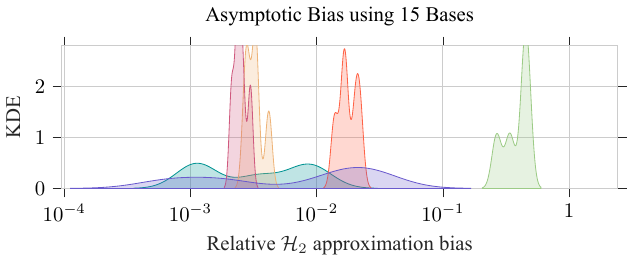}
    \end{subfigure}
    \caption{\res{Asymptotic relative $\mathcal H_2$ approximation bias of the nominal system
    \eqref{eq:small_system_equation} under Gaussian perturbations of the true system
    poles, visualized using a multi-kernel density estimate (KDE).}
    }
    \label{fig:small_system_bias}
\end{figure}

Next, to further demonstrate the effectiveness of the proposed method, we examine the identification of a higher-order system. Consider a heat diffusion process~\cite{mo_network_2009} in a $(3\times 3)\mathrm{m}$ square region with internal obstacles, as shown in Fig.~\ref{fig:region_shape}. The two small circles are centered at $(0.75, 2.25)$ and $(2.25, 0.75)$ respectively, each with a radius of $0.1\mathrm{m}$. Additionally, a half-circle obstacle is centered at $(1.5, 0.95)$ with a radius of $0.55\mathrm{m}$.

\begin{figure}[!htbp]
    \vspace{-0.3cm}
    \centering
    \includegraphics[width=0.4\linewidth]{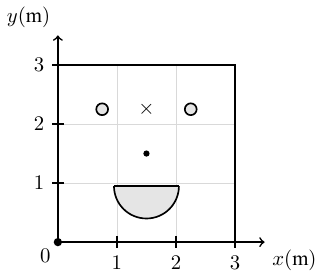}
    \caption{The shape of the region considered in the diffusion process. The obstacles are shown in gray, the heat sources are denoted as $\times$, and the sensor is denoted as the black dot.}
    \label{fig:region_shape}
    \vspace{-0.5cm}
\end{figure}

We denote the temperature at $(x, y)$ and time $t$ by $s(x, y, t)$, and the dynamics of the diffusion process in the square region are characterized by the following Partial Differential Equation (PDE):
\begin{equation}\label{eq:diffusion_pde}
    \frac{\partial s}{\partial t} = \alpha(x, y) \left( \frac{\partial^2 s}{\partial x^2} + \frac{\partial^2 s}{\partial y^2} \right),
\end{equation}
with the boundary condition
\[
    s(x, y, t) = 0, \quad \forall (x, y) \in \mathcal{B},
\]
where $\alpha(x, y)$ denotes the diffusion constant at $(x, y)$, and $\mathcal{B}$ represents the boundary of the region. A heat source is located at $(1.5, 2.25)$ and is set to a temperature $u_k \sim \mathcal{N}(0, 1)$ at each time step. A sensor is positioned at $(1.5, 1.5)$. The system is further discretized on a $10 \times 10$ grid following the method in~\cite{mo_network_2009}. The experiment is repeated $100$ times, with each instance introducing element-wise Gaussian perturbations to the system matrix $A$ (standard deviation of $10^{-3}$) to account for factors like medium heterogeneity. 

We apply \res{all six} pole selection methods used in the previous experiment for simulation, keeping the same configurations, except that the pole region is now set to $\mathcal{D} = [-0.99, 0.99]$. Although $\mathcal D$ is specified as a real interval, the true poles of the system can be complex and lie outside $\mathcal D$ due to the perturbations in $A$, making the \emph{a priori} information inaccurate. 

\res{The simulation results are shown in Fig.~\ref{fig:large_system_bias}. 
For this high-dimensional system, OBFs with poles selected by the minimax, SQP, RA,
and greedy methods exhibit relatively large asymptotic biases, together with
increased variability across trials. In contrast, the Tsuji points achieve the smallest asymptotic bias
using only $15$ bases for the $100$-dimensional system and exhibit consistently robust performance, further
highlighting the effectiveness of the proposed approach in high-dimensional
settings.}

\res{Compared with the lower-dimensional case, the advantage of the Tsuji-based
pole selection becomes more pronounced in the high-dimensional setting.
In low-dimensional systems, small approximation errors may occasionally be observed
due to incidental alignment between the selected basis poles and the dominant poles
of the true system.
As the system dimension increases, such favorable alignment becomes increasingly
unlikely, and the consistently small average and worst-case asymptotic biases
achieved by the Tsuji points highlight their robustness, especially in high-dimensional systems.}

\res{Additionally, we compare the computation time of all pole selection methods as a function of the number of bases~$q$, with the results reported in Fig.~\ref{fig:time_with_q}. The initial guess of the Tsuji points exhibits the shortest computation time among all approaches, while the RA method also demonstrates competitive efficiency. The Tsuji points method is slightly slower but remains in the same order of magnitude. Notably, the runtimes of these three methods are all below $10^{-2}$ seconds and do not increase noticeably as~$q$ grows. In contrast, the remaining minimax, SQP and greedy methods incur higher computational costs and exhibit an increase in runtime for larger~$q$. Overall, these results indicate that the proposed Tsuji pole selection achieves a favorable trade-off between approximation accuracy and computational efficiency, particularly in regimes with larger numbers of bases.}

To further validate the convergence rate derived in Theorem~\ref{thm:convergence}, we solve the least-squares problem in~\eqref{eq:finite_least_squares_1} to identify both previously considered systems. We use $10$ OBFs, each configured with Tsuji points computed for the respective \res{pole region}. A total of $100$ experiments are conducted with $500$ time steps each. The relative $\mathcal{H}_2$ identification bias as a function of time step is plotted in Fig.~\ref{fig:online_error}, confirming the convergence rates in Theorem~\ref{thm:convergence}. \res{For comparison, Fig.~\ref{fig:online_error} also reports the identification error of the Ho-Kalman algorithm, implemented as in~\cite{oymak_revisiting_2022} with system dimension $n=10$ and $T=30$. The results show that the Ho-Kalman algorithm exhibits considerably larger finite-sample identification error for both systems, consistent with results reported in the literature~\cite{chiuso_ill-conditioning_2004, hachicha2014n4sid}. This further illustrates that the OBF-based method achieves more reliable identification performance in finite-sample regimes.}

\begin{figure}[!htbp]
    \centering
    \begin{subfigure}{0.45\textwidth}
        \centering
        \includegraphics[width=\textwidth]{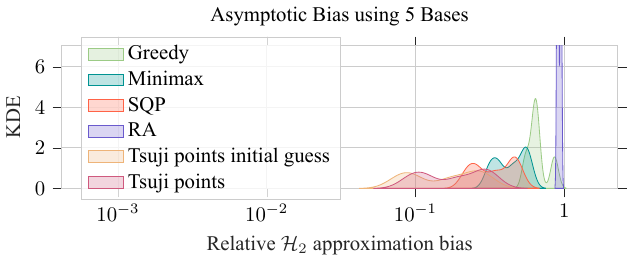}
    \end{subfigure}
    \begin{subfigure}{0.45\textwidth}
        \centering
        \includegraphics[width=\textwidth]{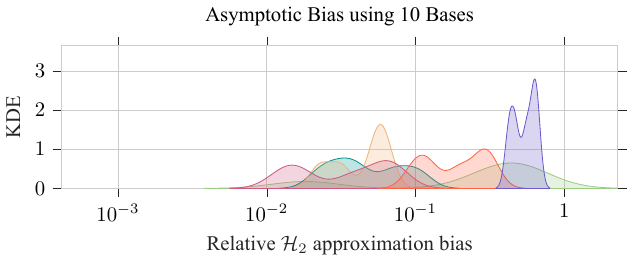}
    \end{subfigure}
    \begin{subfigure}{0.45\textwidth}
        \centering
        \includegraphics[width=\textwidth]{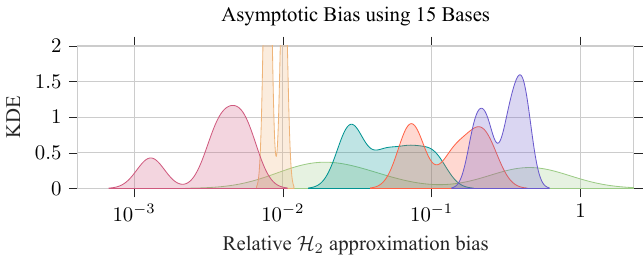}
    \end{subfigure}
    \caption{\res{Asymptotic relative $\mathcal H_2$ approximation bias of the diffusion process
    under Gaussian perturbations of the true system poles, visualized using a
    multi-kernel density estimate (KDE).}}

    \label{fig:large_system_bias}
\end{figure}
\begin{rese}
\begin{figure}[!htbp]
    \centering
    \begin{subfigure}{0.45\textwidth}
        \centering
        \includegraphics[width=\textwidth]{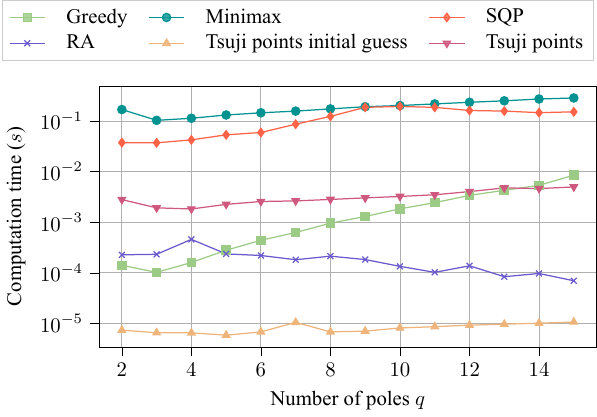}
    \end{subfigure}
    \begin{subfigure}{0.45\textwidth}
        \centering
        \includegraphics[width=\textwidth]{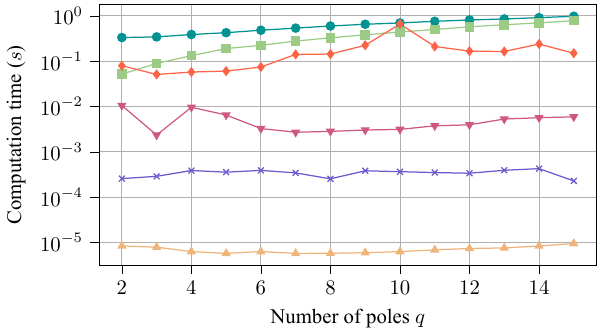}
    \end{subfigure}
    \caption{\res{Computation time versus the number of basis poles $q$ for different pole
    selection algorithms, evaluated on pole regions $\mathcal D=[-0.95, 0.95]$ (top) and $\mathcal D=[-0.99, 0.99]$ (bottom).
    }}
    \label{fig:time_with_q}
    \vspace{-0.5cm}
\end{figure}
\end{rese}
\begin{figure*}[!htbp]
    \centering
    \begin{subfigure}{0.41\textwidth}
        \centering
        \includegraphics[width=\textwidth]{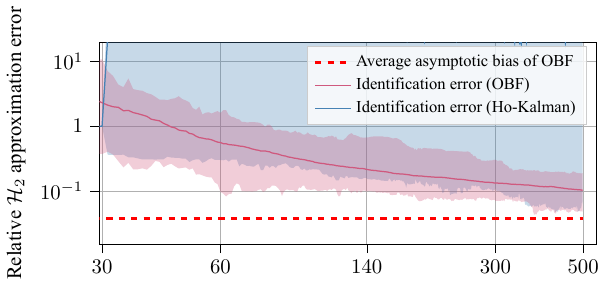}
    \end{subfigure}
    \begin{subfigure}{0.43\textwidth}
        \centering
        \includegraphics[width=\textwidth]{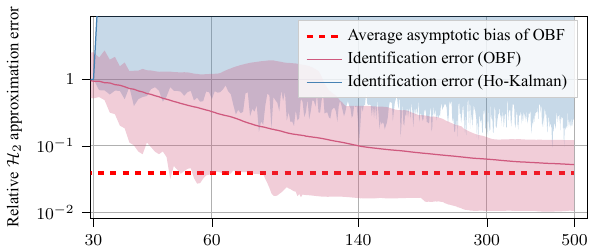}
    \end{subfigure}
    \caption{\res{Online relative $\mathcal H_2$ identification error of the
    system~\eqref{eq:small_system_equation} (left) and the diffusion process (right).
    Results are shown for OBF-based identification using $10$ basis functions with
    Tsuji points, as well as for the Ho-Kalman algorithm.}
    The plot reports $100$ Monte Carlo experiments, where the solid line denotes the
    mean identification error at each time step and the shaded area indicates the
    range of errors across experiments.}
    \label{fig:online_error}
    \vspace{-0.5cm}
\end{figure*}

\section{Conclusion}\label{sec:conclusion}
This paper analyzes the performance of the OBF method, showing that the identification error under $\mathcal H_2$ norm using the OBF method with $N$ samples converges to an asymptotic bias almost surely at the rate of $O(N^{-0.5})$. Additionally, we establish an upper bound on this bias, given by $\bar\alpha \tau(\boldsymbol{\lambda}, \boldsymbol{\mu})$, where $\tau(\boldsymbol{\lambda}, \boldsymbol{\mu})$ denotes the distance between the true system poles $\boldsymbol{\lambda}$ and the OBF poles $\boldsymbol{\mu}$. While the bound suggests that the ideal choice of OBF poles would match the true system poles, we prove that identifying the true poles becomes exponentially challenging as system dimension increases. To address this, we propose the Tsuji pole selection algorithm, which minimizes the worst-case identification bias across a specified class of systems. We further demonstrate that these selected poles achieve a fundamental bound on worst-case identification bias and provide an algorithm to compute near-optimal initial points for the maximization problem, mitigating issues of local minima. Numerical results validate the derived bounds and demonstrate the effectiveness of the proposed Tsuji pole selection algorithm.

\appendices

\section{Proof of Lemma~\ref{lemma:hyperbolic_chebyshev_constant}}\label{append:hyperbolic_chebyshev_constant}
We first introduce the following two lemmas:
\begin{lemma}[Hyperbolic Chebyshev Constant of real intervals; \res{see page~278 in}~\cite{KIRSCH2005243}]\label{lemma:real_chebyshev_constant}
    The hyperbolic Chebyshev constant of the real interval $[0, \rho], \rho<1$ is
    \[\tau([0, \rho])=\exp \left\{-\frac{\pi}{2} \frac{K\left(\sqrt{1-\rho^2}\right)}{K(\rho)}\right\},\]
    where
    \[K(\rho)=\int_0^1 \frac{d x}{\sqrt{\left(1-x^2\right)\left(1-\rho^2 x^2\right)}}.\]
\end{lemma}

\begin{lemma}[Invariance under conformal mapping; \res{see page 278 in}~\cite{KIRSCH2005243}]\label{lemma:conformal_invariant}
    Let $\mathcal{D}$ and $\mathcal{\tilde D}$ be two distinct closed subsets of the open unit disk $\mathbb{D}$. Let $\mathcal{F}$ be the family of conformal mapping $f$ of $\mathbb{D}\backslash\mathcal{D}$ onto $\mathbb{D}\backslash\mathcal{\tilde D}$ bordering on the unit circle, such that $f(\partial \mathbb{D})=\partial \mathbb{D}$. Then
    \[\tau(\mathcal{\tilde D})=\tau(\mathcal{D})\]
    for all $f\in\mathcal{F}$.
\end{lemma}

Then, we are ready to prove the results in Lemma~\ref{lemma:hyperbolic_chebyshev_constant}.
\begin{proof}
    Statement~\ref{item:disk} is proved by~\cite{KIRSCH2005243} and is omitted here for simplicity.

    For statement~\ref{item:real_interval}, according to Lemma~\ref{lemma:conformal_invariant}, we aim to seek the corresponding conformal mapping that maps the interval $[\rho_1, \rho_2]$ to $[0, \tilde \rho]$. Then, Lemma~\ref{lemma:real_chebyshev_constant} can be applied to obtain the hyperbolic Chebyshev constant of the interval $[\rho_1, \rho_2]$. One can verify that the following conformal mapping satisfies our requirements:
    \[f(z)=\frac{z-\rho_1}{1-\bar\rho_1z},\]
    and $f(\rho_1)=0$. Therefore,
    \[\tilde \rho=f(\rho_2)=\frac{\rho_2-\rho_1}{1-\rho_1\rho_2}.\]
    Finally, using Lemma~\ref{lemma:real_chebyshev_constant}, statement~\ref{item:real_interval} is proved.
\end{proof}

\section{Proof of Theorem~\ref{cor:system_approx_error} and Theorem~\ref{thm:system_approximation_error}}\label{append:system_approx_error}
We briefly outline our main idea for quantifying the asymptotic bias in Theorem~\ref{cor:system_approx_error}. The bias between \(G(z)\) and \(\check G^*(z)\) is influenced by two factors: the statistical properties of the input signal \(u_t\) and the approximation bias of the system \(G(z)\) using the bases \(V_k(z)\). To analyze the asymptotic bias, we first isolate the effect of \(u_t\) by introducing a medium system $\tilde G(z)$ where \(u_t\) is i.i.d. white noise with zero mean and unit covariance:
\begin{align}
    &\tilde R_1^*, \dots, \tilde R_q^* = \arg\min_{\tilde{R}_1, \dots, \tilde{R}_q} \left\| G(z) - \sum_{k=1}^q \tilde{R}_k V_k(z) \right\|_2,
    \label{eq:R_star_0_definition}
\end{align}
\begin{align}
    \tilde G(z) &= \sum_{k=1}^q \tilde{R}_k^* V_k(z), \quad G_e(z) = G(z) - \tilde G(z). \label{eq:G_e_def}
\end{align}
The system \(\tilde G(z)\) also serves as the optimal solution to Problem~\ref{prob:system_identification}. Using $\tilde G(z)$, we can now quantify the bias between $\check G^*(z)$ and $G(z)$ by first establishing the relationship between \(\|\check G^*(z) - G(z)\|_2\) and \(\|G_e(z)\|_2\), which captures the influence of the input signal \(u_t\). We then proceed to quantify the asymptotic bias \(\|G_e(z)\|_2\).

The following proposition provides an upper bound on the asymptotic bias between $\check G^*(z)$ and $G(z)$ in terms of $G_e(z)$:
\begin{proposition}[Asymptotic bias with arbitrary inputs\res{; following Proposition~6.3 in}~\cite{van_den_hof_system_1995}]\label{thm:bias_with_arbitrary_inputs}
    Suppose the noise sequence $\{v_k\}$ is a martingale difference sequence. Then, the asymptotic approximation bias between the solution of the least-squares problem~\eqref{eq:finite_least_squares_1}, $\check G^*(z)$, and the true system $G(z)$ satisfies
    \begin{align}
        \left\|\check G^*(z)-G(z)\right\|_2\leq \left[1+\frac{\mathrm{ess}\sup_{\omega}\|\Phi_u(\omega)\|}{\mathrm{ess}\inf_{\omega}\|\Phi_u(\omega)\|}\right]\|G_e(z)\|_2,\label{eq:approximation_bias_arbitrary_inputs}
    \end{align}
    where $G_e(z)$ is defined in~\eqref{eq:G_e_def} and $\check G^*(z)$ is introduced in~\eqref{eq:asymptotic_result}.
\end{proposition}

Proposition~\ref{thm:bias_with_arbitrary_inputs} isolates the effect of the input signal \(u_t\) on the approximation bias, enabling us to focus on \(G_e(z)\). The following lemma provides an upper bound on \(\|G_e(z)\|_2\).

\begin{lemma}[Optimal $\mathcal H_2$ approximation bias with fixed $\mu_k$]\label{thm:system_approx_error}
    For a fixed group of bases $\{1/(z-\mu_k)\}_{k=1}^q$, the optimal $\mathcal H_2$ approximation bias $G_e(z)$ of $G$ using any linear combination of these bases is bounded by
    \begin{align}
        \|G_e(z)\|_2\leq \sum_{j=1}^n \frac{\|R_j\|}{\sqrt{1-|\lambda_j|^2}}\prod_{k=1}^q [\lambda_j, \mu_k]_{\mathrm h},
    \end{align}
    where $R_j$ and $\lambda_j$ are defined in~\eqref{eq:tf_decomposition}.
\end{lemma}

To prove Lemma~\ref{thm:system_approx_error}, we first prove the following theorem:
\begin{theorem}\label{thm:vector_error}
    Let scalars $r_1, \cdots, r_q\in\mathbb{C}$. For an arbitrary $\lambda\in\mathbb{D}, \forall \mu_1, \cdots, \mu_q\in\mathbb{D}$ and $\mu_j\neq \mu_k, \forall j\neq k$, then the relative approximation error of $\frac{1}{z-\lambda}$ using the linear combination of $\frac{1}{z-\mu_1}, \cdots, \frac{1}{z-\mu_q}$ satisfies
    \begin{equation}\label{eq:vector_approx_general}
        \min_{r_1, \cdots, r_q}\frac{\left\|\frac{1}{z-\lambda}-\sum_{k=1}^{q}\frac{r_k}{z-\mu_k}\right\|_2}{\left\|\frac{1}{z-\lambda}\right\|_2}=\prod_{k=1}^q [\lambda, \mu_k]_{\mathrm{h}},
    \end{equation}
    where $[\cdot, \cdot]_{\mathrm{h}}$ denotes the pseudohyperbolic metric defined in~\eqref{eq:pseudohyperbolic_metric}.
\end{theorem}

\begin{proof}
    The case where $\exists k, \lambda=\mu_k$ is trivially true since both sides of~\eqref{eq:vector_approx_general} are $0$. Thus, we only consider the situation when $\lambda\neq \mu_k, \forall k=1, \cdots, q$.
    We first consider the case where $\lambda\neq 0$, and $\mu_k\neq 0, \forall k=1, \cdots, q$. The least-squares error can be written as:
    \begin{equation}\label{eq:vector_approx_error}
        \min_{r_1, \cdots, r_q}\left\|\frac{1}{z-\lambda}-\sum_{k=1}^q \frac{r_k}{z-\mu_k}\right\|_2^2=\phi_{\lambda, \lambda}-p_{\lambda, \boldsymbol{\mu}}\Xi_{\boldsymbol{\mu}, \boldsymbol{\mu}}^{-1}p_{\lambda, \boldsymbol{\mu}}^H,
    \end{equation}
    where
    \[\phi_{\lambda, \lambda}=\langle\frac{1}{z-\lambda}, \frac{1}{z-\lambda}\rangle=\frac{1}{1-|\lambda|^2}, \]
    and similarly,
    \[p_{\lambda, \boldsymbol{\mu}}=\begin{bmatrix}
        \frac{1}{1-\lambda\bar\mu_1} & \frac{1}{1-\lambda\bar\mu_2 } & \cdots & \frac{1}{1-\lambda\bar\mu_q}
    \end{bmatrix},\]
    and $\Xi_{\boldsymbol{\mu}, \boldsymbol{\mu}}$ is defined in~\eqref{eq:xi_mu_mu}. Denote
    \[ \check \Xi\triangleq \begin{bmatrix}
        \phi_{\lambda, \lambda} & p_{\lambda, \boldsymbol{\mu}} \\ p_{\lambda, \boldsymbol{\mu}}^H & \Xi_{\boldsymbol{\mu}, \boldsymbol{\mu}}
    \end{bmatrix}.\]
    Then one can verify that
    \begin{equation}\label{eq:check_Xi}
        \check\Xi=D_{\lambda, \boldsymbol{\mu}}\Xi_{\lambda, \boldsymbol{\mu}},
    \end{equation}
    where
    \[D_{\lambda, \boldsymbol{\mu}}=\text{diag}\left(\frac{1}{\lambda}, \frac{1}{\mu_1}, \cdots, \frac{1}{\mu_q}\right),\]
    \[\Xi_{\lambda, \boldsymbol{\mu}}=\begin{bmatrix} \frac{1}{1/\lambda-\bar\lambda} & \frac{1}{1/\lambda-\bar\mu_1} & \cdots & \frac{1}{1/\lambda-\bar\mu_q} \\
    \frac{1}{1/\mu_1-\bar\lambda} & \frac{1}{1/\mu_1-\bar\mu_1} & \cdots & \frac{1}{1/\mu_1-\bar\mu_q} \\
    \vdots & \vdots & \ddots & \vdots \\
    \frac{1}{1/\mu_q-\bar\lambda} & \frac{1}{1/\mu_q-\bar\mu_1} & \cdots & \frac{1}{1/\mu_q-\bar\mu_q}\end{bmatrix}.\]
    Notice that $\Xi_{\lambda, \boldsymbol{\mu}}$ is a Cauchy matrix, and is invertible since $\mu_k$ are distinct from each other and $\lambda\neq \mu_k$. Moreover, by block matrix inversion lemma, the top left entry of $\check\Xi^{-1}$
    is $[\check\Xi^{-1}]_{1, 1}=(\phi_{\lambda, \lambda}-p_{\lambda, \boldsymbol{\mu}}\Xi_{\boldsymbol{\mu}, \boldsymbol{\mu}}^{-1} p_{\lambda, \boldsymbol{\mu}}^H)^{-1}$. Thus, we can obtain the LHS in the theorem by considering the inversion of $\check \Xi$:
    \begin{equation}\label{eq:epsilon_lambda}
        \begin{aligned}
            &\phi_{\lambda, \lambda}-p_{\lambda, \boldsymbol{\mu}}\Xi_{\boldsymbol{\mu}, \boldsymbol{\mu}}^{-1} p_{\lambda, \boldsymbol{\mu}}^H \\
            &=\frac{\prod_{k=1}^q\left(1/\lambda-1/\mu_k\right)\left(\bar\mu_k-\bar\lambda\right)}{(1-|\lambda|^2)\prod_{k=1}^q (1/\lambda-\bar\mu_k)(1/\mu_k-\bar\lambda)} \\
            &=\frac{1}{1-|\lambda|^2}\prod_{k=1}^q\left|\frac{\mu_k-\lambda}{1-\bar\mu_k\lambda}\right|^2.
        \end{aligned}
    \end{equation}
    Equation~\eqref{eq:epsilon_lambda} leverages the explicit form of the Cauchy matrix inversion~\cite{knuth1997art}.

Next, we consider the case where $\lambda=0$ or $\mu_k=0$, but $\lambda$ and $\mu_k$ are still distinct from each other. In this case, the decomposition~\eqref{eq:check_Xi} no longer exists. However, one can prove that the matrix $\Xi_{\boldsymbol{\mu}, \boldsymbol{\mu}}$ and $\check\Xi$ are still invertible. Thus, the error equation~\eqref{eq:vector_approx_error} still holds. Moreover, since matrix inverse $\check\Xi^{-1}$ is differentiable w.r.t. matrix elements, we can see that $[\check\Xi^{-1}]_{1, 1}$ is \emph{continuous} w.r.t. $\lambda$ and $\mu_k$. Thus, the result in~\eqref{eq:epsilon_lambda} still holds when $\lambda=0$ or $\mu_k=0$.
\end{proof}
\vspace{-\baselineskip}%
\begin{remark}
    \res{Equation~(4.48) in~\cite{van_den_hof_system_2005} provides a pointwise bound on the frequency-response error
    $|G_e(e^{i\omega})|$ on the unit
    circle. Integrating this pointwise bound over the unit
    circle yields an $\mathcal{H}_2$-type bound that is equivalent to
    Lemma~\ref{thm:system_approx_error}. Related 
    continuous-time variants are also discussed in~\cite{Brinsmead2001}.} Here, we introduce an alternative proof that utilizes the inner product and the inversion of the Cauchy matrix to derive the discrete-time result.
\end{remark}

\begin{remark}\label{remark:repeated_poles}
    The approximation bias of $1/(z-\lambda)^l, l>1$ using $1/(z-\mu_k), k=1, \cdots, q$ can be similarly derived using this method. Essentially, the difference lies in the form of the covariance matrix $\check \Xi$, which also contains the partial derivative of $\lambda$ to the current $\check \Xi$. Thus, the explicit form of the matrix inverse $\check \Xi^{-1}$ and the resulting approximation error may be further derived for a fixed $l$.
\end{remark}

Next, we are ready to prove Lemma~\ref{thm:system_approx_error}.
\begin{proof}
    The objective function of Problem~\ref{prob:system_identification} has the following upper bound:
    \begin{equation}\label{eq:optimize_H2_detail}
        \begin{aligned}
            &\left\|G_e(z)\right\|_2 \\
            &\leq \min_{\tilde r_{11}, \cdots, \tilde r_{1n}, \cdots, \tilde r_{qn}}\sum_{j=1}^n \left(\|R_j\|\left\|\frac{1}{z-\lambda_j}-\sum_{k=1}^q \frac{\tilde r_{kj}}{z-\mu_k}\right\|_2\right),
        \end{aligned}
    \end{equation}
    where $\tilde r_{kj}$ are scalar optimization variables. The inequality holds by restricting $\check R_k=\sum_{j=1}^n \tilde r_{kj}R_j$.

    Note that the optimization variables $\tilde r_{kj}$ are independent of each other. Thus, we can minimize each term for a given $j$ separately. By Theorem~\ref{thm:vector_error}, for a given $j=1, \cdots, n$,
    \begin{equation}
        \begin{aligned}
            \min_{\tilde r_{1j}, \cdots, \tilde r_{qj}}&\left\|\frac{1}{z-\lambda_j}-\sum_{k=1}^q\frac{\tilde r_{kj}}{z-\mu_k}\right\|_2 \\
            &\qquad\qquad\qquad=\frac{1}{\sqrt{1-|\lambda_j|^2}}\left(\prod_{k=1}^q [\lambda_j, \mu_k]_{\mathrm{h}}\right).
        \end{aligned}
    \end{equation}

    Thus, the approximation error in Lemma~\ref{thm:system_approx_error} satisfies
    \begin{equation*}
        \begin{aligned}
            \|G_e(z)\|_2\leq \sum_{j=1}^n \frac{\|R_j\|}{\sqrt{1-|\lambda_j|^2}}\left(\prod_{k=1}^q [\lambda_j, \mu_k]_{\mathrm{h}}\right).
        \end{aligned}
    \end{equation*}
\end{proof}
Finally, we prove the results in Theorem~\ref{thm:system_approximation_error}. 
\begin{proof}
By the result in Lemma~\ref{thm:system_approx_error}, we have that
\begin{align}
    &\max_{G\in\mathscr{G}}\min_{\check R_1, \cdots, \check R_q}\left\|G-\sum_{k=1}^q\frac{\check R_k}{z-\mu_k}\right\|_2\nonumber \\
    &\leq \left[\max_{\lambda\in\mathcal D}\frac{1}{\sqrt{1-|\lambda|^2}}\left(\prod_{k=1}^q [\lambda, \mu_k]_{\mathrm{h}}\right)\right]\left(\sum_{j=1}^n \|R_j\|\right)\nonumber \\
    &\leq\max_{\lambda\in\mathcal D}\left[\frac{\bar R}{\sqrt{1-|\lambda|^2}}\left(\prod_{k=1}^q [\lambda, \mu_k]_{\mathrm{h}}\right)\right],\label{eq:max_min_error}
\end{align}
where $\bar R\geq 0$ denotes the upper bound on the modified system energy of systems in $\mathscr G$ defined before Theorem~\ref{thm:system_approximation_error}.
Moreover, by choosing
\[G^*(z)=\frac{R^*}{z-\lambda^*}\in\mathscr G,\]
where $R^*=\begin{bmatrix} \bar R & 0 \\ 0 & 0\end{bmatrix}\in\mathbb{R}^{p\times m}$ and $\lambda^*\in\arg\max_{\lambda\in\mathcal D}\frac{1}{\sqrt{1-|\lambda|^2}}\left(\prod_{k=1}^q [\lambda, \mu_k]_{\mathrm{h}}\right)$,
the approximation error in~\eqref{eq:max_min_error} achieves the upper bound:
\[\min_{\check R_1, \cdots, \check R_q}\|G^*-\check G\|_2=\max_{\lambda\in\mathcal D}\frac{\bar R}{\sqrt{1-|\lambda|^2}}\left(\prod_{k=1}^q [\lambda, \mu_k]_{\mathrm{h}}\right),\]
where $\check G(z)=\sum_{k=1}^q \frac{\check R_k}{z-\mu_k}$. Therefore,
\begin{align}
    &\min_{\mu_1, \cdots, \mu_q\in\mathcal D}\max_{G\in\mathscr{G}}\min_{\check R_1, \cdots, \check R_q}\|G-\check G\|_2^{1/q}\nonumber \\
    &=\min_{\mu_1, \cdots, \mu_q\in\mathcal D}\max_{\lambda\in\mathcal D}\frac{\bar R^{1/q}}{(1-|\lambda|^2)^{1/2q}}\left(\prod_{k=1}^q [\lambda, \mu_k]_{\mathrm{h}}\right)^{1/q}.\label{eq:min_max_min_equality}
\end{align}
Moreover, one can verify that
\begin{align}
    &\bar R^{1/q}\min_{\mu_1, \cdots, \mu_q\in\mathcal D}\max_{\lambda\in\mathcal D}\left(\prod_{k=1}^q [\lambda, \mu_k]_{\mathrm{h}}\right)^{1/q}\nonumber \\
    &\leq \min_{\mu_1, \cdots, \mu_q\in\mathcal D}\max_{\lambda\in\mathcal D}\frac{\bar R^{1/q}}{(1-|\lambda|^2)^{1/2q}}\left(\prod_{k=1}^q [\lambda, \mu_k]_{\mathrm{h}}\right)^{1/q}\nonumber \\
    &\leq \frac{\bar R^{1/q}}{(1-\rho_\lambda^2)^{1/2q}}\min_{\mu_1, \cdots, \mu_q\in\mathcal D}\max_{\lambda\in\mathcal D}\left(\prod_{k=1}^q [\lambda, \mu_k]_{\mathrm{h}}\right)^{1/q},\nonumber
\end{align}
where $\rho_\lambda=\max_{\lambda\in\mathcal D}|\lambda|$. By taking limits on all sides of the inequality as $q\to\infty$, and applying the definition of the hyperbolic Chebyshev constant in Definition~\ref{def:chebyshev_constant}, we can prove that
\[\lim_{q\to\infty}\min_{\mu_1, \cdots, \mu_q\in\mathcal D}\max_{G\in\mathscr G}\min_{\check R_1, \cdots, \check R_q}\left\|G-\sum_{k=1}^q\frac{\check R_k}{z-\mu_k}\right\|_2^{1/q}=\tau(\mathcal D).\]
As a result, for any $\nu_1, \cdots, \nu_q\in\mathcal D$,
\begin{align}
    &\liminf_{q\to\infty}\max_{G\in\mathscr G}\min_{\check R_1, \cdots, \check R_q}\left\|G-\sum_{k=1}^q \frac{\check R_k}{z-\nu_k}\right\|_2^{1/q}\nonumber \\
    &\geq \lim_{q\to\infty}\min_{\mu_1, \cdots, \mu_q\in\mathcal D}\max_{G\in\mathscr G}\min_{\check R_1, \cdots, \check R_q}\left\|G-\sum_{k=1}^q\frac{\check R_k}{z-\mu_k}\right\|_2^{1/q}\nonumber \\
    &=\tau(\mathcal D).\nonumber
\end{align}
The second statement in Theorem~\ref{thm:system_approximation_error} can be proved by combining Theorem~\ref{thm:tsuji_asymptotic_optimal} with~\eqref{eq:min_max_min_equality}.
\end{proof}

\section{Conformal Mapping of Real Intervals}\label{append:tsuji_points}
When the region of poles $\mathcal{D}$ is a closed real interval $\mathcal{D}=[-\rho, \rho], \rho<1$, the conformal mapping $g(z)$ that maps the region $\mathbb{D}\backslash \mathcal{D}$ to the annulus $\{w:\tau(\mathcal{D})<|w|<1\}$ can be computed explicitly, where $\tau(\mathcal{D})$ denotes the hyperbolic Chebyshev constant of the closed inverval $\mathcal{D}$.
\begin{theorem}[Conformal mapping from a real interval to an annulus]\label{thm:conformal_mapping}
    The real interval $[-\rho, \rho]$ can be conformally mapped to the annulus $\{w:\tau(\mathcal{D})<|w|<1\}$ by the composition of the following conformal mappings:
    \begin{itemize}
        \item \textbf{M\"obius transform}:
        \begin{equation}\label{eq:mobius_transform}
            m(z)=\frac{z+\rho}{1+\rho z}.
        \end{equation}
        \item \textbf{Schwarz-Christoffel mapping}:
        \begin{equation}\label{eq:schwarz_christoffle_mapping}
            s(z)=\int_0^{\sqrt{\frac{z}{\tilde \rho}}} \frac{1}{\sqrt{(1-w^2)(1-\tilde \rho^2 w^2)}}dw,
        \end{equation}
        where $\tilde \rho=\frac{2\rho}{1+\rho^2}$.
        \item \textbf{Translation and exponential transformations}:
        \[t(z)=-iz+0.5iK(\sqrt{1-\tilde \rho^2})+K(\tilde \rho),\]
        \begin{equation}\label{eq:exp_transform}
            \epsilon(z)=\exp\left(\frac{\pi z}{K(\tilde \rho)}\right).
        \end{equation}
    \end{itemize}
\end{theorem}
\begin{proof}
    First, the M\"obius transform $m(z)$ defined in~\eqref{eq:mobius_transform} maps the region $\mathbb{D}\backslash\mathcal{D}$ to the region $\mathbb{D}\backslash[0, \frac{2\rho}{1+\rho^2}]=\mathbb{D}\backslash[0, \tilde \rho]$.

    Then, by reflecting the region $\mathbb{D}\backslash [0, \tilde \rho]$ at $|z|=1$, we obtain the quadrilateral $Q(\infty, 0, \tilde \rho, 1/\tilde \rho)$ formed from the upper half-plane. According to~\cite{lehto_verzerrungssatze_1965}, the following Schwarz-Christoffel mapping
    \begin{equation}\label{eq:schwarz_christoffle_mapping_1}
        \tilde s(z)=\frac{1}{2i\sqrt{\tilde \rho}}\int_0^z \frac{dw}{\sqrt{w(w-\tilde \rho)(w-1/\tilde \rho)}}
    \end{equation}
    maps the region $\mathbb{D}\backslash [0, \tilde\rho]$ to the rectangle with the vertices $(0, K(\sqrt{1-\tilde \rho^2}), K(\sqrt{1-\tilde \rho^2})-K(\tilde\rho)i, -K(\tilde\rho)i)$. Note that by substituting $w=\tilde \rho \tilde w^2$ in the integral,~\eqref{eq:schwarz_christoffle_mapping_1} becomes
    \[\tilde s(z)=-i\int_0^{\sqrt{\frac{z}{\tilde \rho}}}\frac{d\tilde w}{\sqrt{(1-\tilde w^2)(1-\tilde \rho^2\tilde w^2)}}.\]
    For the simplicity of the inverse mapping, we move the coefficient $-i$ into the translation step. Thus, the transformation in~\eqref{eq:schwarz_christoffle_mapping} maps the region $\mathbb{D}\backslash [0, \tilde \rho]$ to the rectangle with the vertices $(0, K(\tilde\rho), iK(\sqrt{1-\tilde \rho^2})+K(\tilde\rho), iK(\sqrt{1-\tilde \rho^2}))$.

    To obtain the unique conformal mapping, we ensure $g(1)=1$ by translating the rectangle by $t(z)$ and then applying the exponential transformation $\epsilon(z)$ defined in~\eqref{eq:exp_transform} to obtain the annulus $\{w:\tau(\mathcal{D})<|w|<1\}$.

    The conformal mapping from the region $\mathbb{D}\backslash [-\rho, \rho]$ to the annulus $\{w:\tau([-\rho, \rho])<w<1\}$ is visualized in Fig.~\ref{fig:conformal_mapping}.
\end{proof}
\begin{figure*}[!htbp]
    \centering
    \includegraphics[width=\textwidth]{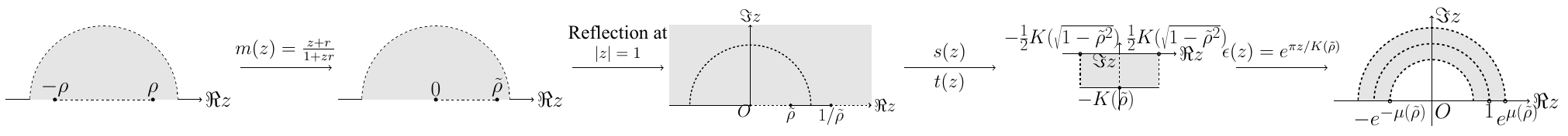}
    \caption{Conformal mapping from the real interval $[-\rho, \rho]$ to the annulus $\{w:\tau([-\rho, \rho])<w<1\}$.}
    \label{fig:conformal_mapping}
\end{figure*}

Therefore, we can find the points $\{\eta_{qk}\}_{k=1}^q$ of the real interval $[-\rho, \rho]$ by the following steps:
\begin{itemize}
    \item Uniformly sample $q$ points from the circle $\{z\in\mathbb{D}\mid |z|=\tau([-\rho, \rho])\}$ on the complex plane.
    \item Transform each point using the following mappings:
    \begin{itemize}
        \item Inverse of \textit{exponential transformation and translation}:
        \[\epsilon^{-1}(z)=\frac{K(\tilde \rho)}{\pi}\log(z),\]
        \[t^{-1}(z)=iz+0.5K(\sqrt{1-\tilde \rho^2})-iK(\tilde \rho).\]
        \item Inverse of the Schwarz-Christoffel mapping:
        \[s^{-1}(z)=\tilde \rho\ \mathrm{sn}(z, \tilde \rho^2)^2, \]
        where $\mathrm{sn}(u, m)$ denotes the \textit{Jacobi elliptic sine function}, which is the inverse function of the incomplete elliptic integral of the first kind.
        \item Inverse of the Mobius transform:
        \[m^{-1}(z)=\frac{z-\rho}{1-z\rho}.\]
    \end{itemize}
\end{itemize}

On the other hand, when the region of poles $\mathcal{D}$ is a continuum on the complex plane, the conformal mapping from the annulus $\{z:\tau(\mathcal{D})<|z|<1\}$ to the region $\mathbb{D}\backslash \mathcal{D}$ can also be numerically computed~\cite{trefethen_numerical_2020}.

\section{Relationship between the state-space model and the transfer function}\label{append:ss_tf}
This section discusses the relationship between the state-space model and the transfer function. First, consider an $n$-dimensional state-space model with $m$ inputs and $p$ outputs satisfying Assumption~\ref{assump:system_stability} and~\ref{assump:bias}:
\begin{equation}\label{eq:true_system_noisy}
    S:\left\{\begin{aligned}
        x_{t+1}&=Ax_t+Bu_t, \\
        y_t &= Cx_t.
    \end{aligned}\right.
\end{equation}
Denote the eigenvalues of $A$ as $\lambda_1, \cdots, \lambda_n$. Since similarity transformation preserves the input-output relationship of the system, we directly assume that the system takes the diagonal canonical form and denote the state-space parameters as:

\begin{proposition}[Markov parameter description and transfer function of a state-space model]\label{prop:markov_parameter_description}
   For a state-space representation along with parameters $A, B, C$, the corresponding transfer function is
   \begin{equation}\label{eq:markov_parameter_ss}
       G(z)=\sum_{j=1}^n \frac{R_j}{z-\lambda_j}.
   \end{equation}
   with $R_j=c_jb_j^\top$,
\end{proposition}

Moreover, given a transfer function with the decomposition $G(z)=\sum_{j=1}^\kappa \frac{R_j}{z-\lambda_j}$, one can find a corresponding state-space realization. Let
\[\varrho\triangleq \sum_{j=1}^\kappa \mathrm{rank}(R_j).\]
\begin{theorem}[State-space realization of Markov parameter description]\label{thm:system_realization}
   The transfer function
   \begin{equation}\label{eq:tf_decomposition_2}
       G(z)=\sum_{j=1}^\kappa \frac{R_j}{z-\lambda_j}
   \end{equation}
   can be realized by a $\varrho$-dimensional state-space model, where $\varrho$ is at most $\min(p, m)\kappa$.
\end{theorem}

\begin{proof}
    Equation~\eqref{eq:markov_parameter_ss} in Proposition~\ref{prop:markov_parameter_description} follows directly from the definition of transfer functions $G(z)=C(zI-A)^{-1}B$.

    The rest of the proof is devoted to the state-space realization in Theorem~\ref{thm:system_realization}. Let
    \[\ell_j\triangleq \mathrm{rank}(R_j), j=1, \cdots, n.\]
    Then, for each $j=1, \cdots, n$, $R_j$ can be decomposed as:
    \[R_j=c_1^{(j)}(b_1^{(j)})^\top+\cdots+c_{\ell_j}^{(j)}(b_{\ell_j}^{(j)})^\top, c_k^{(j)}\in\mathbb{C}^{p}, b_k^{(j)}\in\mathbb{C}^{m},\]
    where each product $c_k^{(j)}(b_k^{(j)})^\top$ and the corresponding vector $c_k^{(j)}, b_k^{(j)}$ can be computed via the Singular Value Decomposition (SVD) of $R_j$. Then, one can verify that the transfer function~\eqref{eq:tf_decomposition_2} can be realized by the following $\varrho$-dimensional state-space model:
    \[A=\mathrm{diag}(\underbrace{\lambda_1, \cdots, \lambda_1}_{\ell_1\ \text{times}}, \cdots, \underbrace{\lambda_n, \cdots, \lambda_n}_{\ell_n\ \text{times}}), \]
    \[B=\begin{bmatrix}
        b_1^{(1)} & \cdots & b_{\ell_1}^{(1)} & \cdots & b_1^{(n)} & \cdots & b_{\ell_n}^{(n)}
    \end{bmatrix}^\top,\]
    \[ C=\begin{bmatrix}
        c_1^{(1)} & \cdots & c_{\ell_1}^{(1)} & \cdots & c_1^{(n)} & \cdots & c_{\ell_n}^{(n)}
    \end{bmatrix}.\]
\end{proof}
\section{Proof of Theorem~\ref{thm:fundamental_limit} and Corollary~\ref{cor:fundamental_limit_ss}}\label{append:fundamental_limit}
We first prove Theorem~\ref{thm:fundamental_limit}.
\begin{proof}
    First, we derive the KL divergence between the two hypotheses. If the samples $\{u_{t-1}, y_t\}_{t=1}^N$ are from the true system $G$, since we assume that the past inputs $u_{-1}, u_{-2}, \cdots$ are $\boldsymbol{0}$, using the expansion $G(z)=H_1z^{-1}+H_2z^{-2}+\cdots$ where $H_k$ are the Markov parameters of the system, we have that
    \[y_t=\sum_{l=1}^{t}H_l u_{t-l}+v_t.\]
    Let $\mathcal{F}_N$ denote the $\sigma$-algebra generated by the inputs $u_0, \cdots, u_{N-1}$.
    If the sample trajectories are from the true system, then
    \begin{equation}\label{eq:true_system_distribution}
        \begin{bmatrix}
            y_1^\top & y_2^\top & \cdots & y_N^\top
        \end{bmatrix}^\top\left|{\mathcal{F}_N}\right.\sim\mathcal{N}\left(\mathcal{G}_N\mathcal{U}_N, I_N\otimes\mathcal{R}\right),
    \end{equation}
    where
    \[\mathcal{G}_N=\begin{bmatrix}
        H_1 & 0 & \cdots & 0 \\
        H_2 & H_1 & \cdots & 0 \\
        \vdots & \vdots & \ddots & \vdots \\
        H_N & H_{N-1} & \cdots & H_1
    \end{bmatrix}, \mathcal{U}_N=\begin{bmatrix}
        u_0 \\ u_1 \\ \vdots \\ u_{N-1}
    \end{bmatrix}.\]
    On the other hand, if the samples are from the surrogate system $\check G$, then
    \begin{equation}\label{eq:surrogate_system_distribution}
        \begin{bmatrix}
            y_1^\top & y_2^\top & \cdots & y_N^\top
        \end{bmatrix}^\top \left|\mathcal{F}_N\right.\sim\mathcal{N}\left(\mathcal{\check G}_N\mathcal{U}_N, I_N\otimes \mathcal{R}\right),
    \end{equation}
    where $\mathcal{\check G}_N$ is similarly defined as $\mathcal{G}_N$, but is composed of $\check H_1, \cdots, \check H_N$ instead.

    Let
    \[L_N=\log\frac{f_{G}(y_1, \cdots, y_N, u_0, \cdots, u_{N-1})}{f_{\check G}(y_1, \cdots, y_N, u_0, \cdots, u_{N-1})},\]
    where $f_{G}$ denotes the probability density function of $\{y_t, u_{t-1}\}_{t=1}^N$ if the samples are from the true system $G$, and the notation $f_{\check G}$ is similarly defined. Then, the KL divergence between the two hypotheses in~\eqref{eq:hypothesis} can be computed as:
    \begin{equation}
        \begin{aligned}
            D_{\mathrm{KL}}(\mathbb{P}_0\|\mathbb{P}_1)&=D_{\mathrm{KL}}(\mathbb{P}_1\|\mathbb{P}_0)=\mathbb{E}(L_N)=\mathbb{E}[\mathbb{E}(L_N|\mathcal{F}_N)] \\
            &=\frac{1}{2}\mathbb{E}[(\Delta\mathcal{G}_N\mathcal{U}_N)^\top (I_N\otimes\mathcal{R}^{-1})\Delta\mathcal{G}_N\mathcal{U}_N],
        \end{aligned}
    \end{equation}
    where $\Delta\mathcal{G}_N=\mathcal{G}_N-\mathcal{\check G}_N$. Since $I_N\otimes \mathcal{R}^{-1}\preceq I_{Np}\otimes \|\mathcal{R}^{-1}\|$, we can further bound the KL divergence as:
    \begin{equation}
        \begin{aligned}
            &D_{\mathrm{KL}}(\mathbb{P}_0\|\mathbb{P}_1)\leq \frac{\|\mathcal{R}^{-1}\|}{2}\mathrm{tr}(\Delta\mathcal{G}_N\mathbb{E}(\mathcal{U}_N\mathcal{U}_N^\top) \Delta\mathcal{G}_N^\top) \\
            &=\frac{\|\mathcal{R}^{-1}\|\mathrm{ess}\sup_{\omega}\|\Phi_u(\omega)\|}{2}\mathrm{tr}(\Delta\mathcal{G}_N\Delta\mathcal{G}_N^\top) \\
            &\leq \frac{\|\mathcal{R}^{-1}\|\mathrm{ess}\sup_{\omega}\|\Phi_u(\omega)\|}{2}\left(\sum_{t=1}^N\|H_t-\check H_t\|^2\right)N \\
            &\leq \frac{\|\mathcal{R}^{-1}\|\mathrm{ess}\sup_{\omega}\|\Phi_u(\omega)\|}{2}\|G-\check G\|^2N \\
            &\leq \frac{\bar R^2\|\mathcal{R}^{-1}\|\mathrm{ess}\sup_{\omega}\|\Phi_u(\omega)\|}{2(1-\rho_\lambda^2)}\tau_{n}(\mathcal{D})^{2 n}N,
        \end{aligned}
    \end{equation}
    where the last step leverages the result in Lemma~\ref{thm:system_approx_error} and the definition of the finite hyperbolic Chebyshev constant. Therefore, to ensure that the KL divergence is no smaller than $\delta$, the sample complexity $N$ should satisfy
    \[N\geq \frac{2\delta(1-\rho_\lambda^2)}{\bar R^2\|\mathcal{R}^{-1}\|\mathrm{ess}\sup_{\omega}\|\Phi_u(\omega)\|}\tau_{n}(\mathcal{D})^{-2 n}.\]

\end{proof}
Next, we focus on Corollary~\ref{cor:fundamental_limit_ss}. By Theorem~\ref{thm:fundamental_limit}, a system $\check G(z)$ with $\underline{n}$ distinct poles can be constructed using the method in Section~\ref{sec:fundamental_limit}, such that the number of samples required to distinguish the true system $G(z)$ from the surrogate system $\check G(z)$ satisfies
\[
N \geq \frac{2\delta(1-\rho_\lambda^2)}{\bar{R}^2 \|\mathcal{R}^{-1}\| \operatorname{ess}\sup_{\omega} \|\Phi_u(\omega)\|} \tau_{\underline n}(\mathcal{D})^{-2 \underline{n}}.
\]
Using Theorem~\ref{thm:system_realization}, the system $\check G(z)$ has a state-space realization $(A, B, C)$ of dimension at most 
\[
\check{\underline{n}} \triangleq \min(p, m) \underline{n} \leq n.
\]
We can extend this realization to dimension $n$ by appending zeros:
\[
\check{A} = 
\begin{bmatrix} 
A & 0 \\ 
0 & 0 
\end{bmatrix} 
\in \mathbb{C}^{n \times n}, \quad 
\check{B} = 
\begin{bmatrix} 
B \\ 
0 
\end{bmatrix} 
\in \mathbb{C}^{n \times m}, \quad 
\]
\[\check{C} = 
\begin{bmatrix} 
C & 0 
\end{bmatrix} 
\in \mathbb{C}^{p \times n}.\]
Thus, the system $(\check{A}, \check{B}, \check{C})$ is an $n$-dimensional system that satisfies Corollary~\ref{cor:fundamental_limit_ss}.

\bibliography{ref}

@article{stiemer_approximation_2005,
	title = {On the approximation order of extremal point methods for hyperbolic minimal energy problems},
	volume = {99},
	issn = {0945-3245},
	url = {https://doi.org/10.1007/s00211-004-0565-2},
	doi = {10.1007/s00211-004-0565-2},
	language = {en},
	number = {3},
	urldate = {2024-01-19},
	journal = {Numerische Mathematik},
	author = {Stiemer, Marcus},
	month = jan,
	year = {2005},
	pages = {533--555},
}

@INPROCEEDINGS{ninness_unifying_1997_cdc,
  author={Ninness, B.M. and Gustafsson, F.},
  booktitle={Proceedings of 1994 33rd IEEE Conference on Decision and Control}, 
  title={A unifying construction of orthonormal bases for system identification}, 
  year={1994},
  volume={4},
  number={},
  pages={3388-3393 vol.4},
  doi={10.1109/CDC.1994.411668}}

@article{ninness_unifying_1997,
	title = {A unifying construction of orthonormal bases for system identification},
	volume = {42},
	issn = {1558-2523},
	url = {https://ieeexplore.ieee.org/abstract/document/566661},
	doi = {10.1109/9.566661},
	number = {4},
	urldate = {2024-01-09},
	journal = {IEEE Transactions on Automatic Control},
	author = {Ninness, B. and Gustafsson, F.},
	month = apr,
	year = {1997},
	pages = {515--521},
}

@book{knuth1997art,
author = {Knuth, Donald E.},
title = {The Art of Computer Programming, Volume 1 (3rd Ed.): Fundamental Algorithms},
year = {1997},
isbn = {0201896834},
publisher = {Addison Wesley Longman Publishing Co., Inc.},
address = {USA}
}

@article{fu1993optimum,
  title={An optimum time scale for discrete Laguerre network},
  author={Fu, Ye and Dumont, Guy Albert},
  journal={IEEE Transactions on Automatic control},
  volume={38},
  number={6},
  pages={934--938},
  year={1993},
  publisher={IEEE}
}

@article{Brinsmead2001,
   author = {T. S. Brinsmead and G. C. Goodwin},
   doi = {10.1109/9.948483},
   issn = {00189286},
   issue = {9},
   journal = {IEEE Transactions on Automatic Control},
   month = {9},
   pages = {1486-1489},
   title = {Fundamental limits in sensitivity minimization: Multiple-input-multiple-output (MIMO) plants},
   volume = {46},
   year = {2001},
}

@article{tsuji1950metrical,
	title        = {{Some metrical theorems on Fuchsian groups}},
	author       = {Masatsugu Tsuji},
	year         = 1950,
	journal      = {Kodai Mathematical Seminar Reports},
	publisher    = {Tokyo Institute of Technology, Department of Mathematics},
	volume       = 2,
	number       = {4-5},
	pages        = {89 -- 93},
	doi          = {10.2996/kmj/1138833951},
	url          = {https://doi.org/10.2996/kmj/1138833951}
}

@ARTICLE{liu2020online,
  author={Liu, Hanxiao and Mo, Yilin and Yan, Jiaqi and Xie, Lihua and Johansson, Karl H.},
  journal={IEEE Transactions on Automatic Control}, 
  title={An Online Approach to Physical Watermark Design}, 
  year={2020},
  volume={65},
  number={9},
  pages={3895-3902},
  doi={10.1109/TAC.2020.2971994}}

@incollection{lehto_verzerrungssatze_1965,
	title        = {Verzerrungssätze für quasikonforme {Abbildungen}},
	author       = {Lehto, O. and Virtanen, K. I.},
	year         = 1965,
	booktitle    = {Quasikonforme {Abbildungen}},
	publisher    = {Springer Berlin Heidelberg},
	address      = {Berlin, Heidelberg},
	pages        = {54--113},
	doi          = {10.1007/978-3-662-42594-7_3},
	isbn         = {978-3-662-42594-7},
	url          = {https://doi.org/10.1007/978-3-662-42594-7_3}
}

@article{wahlberg_system_1991,
	title = {System identification using {Laguerre} models},
	volume = {36},
	issn = {00189286},
	url = {http://ieeexplore.ieee.org/document/76361/},
	doi = {10.1109/9.76361},
	language = {en},
	number = {5},
	urldate = {2024-01-09},
	journal = {IEEE Transactions on Automatic Control},
	author = {Wahlberg, B.},
	month = may,
	year = {1991},
	pages = {551--562},
}

@inproceedings{chen_regularized_2015,
	title = {Regularized system identification using orthonormal basis functions},
	url = {https://ieeexplore.ieee.org/abstract/document/7330716},
	doi = {10.1109/ECC.2015.7330716},
	urldate = {2024-01-09},
	booktitle = {2015 {European} {Control} {Conference} ({ECC})},
	author = {Chen, Tianshi and Ljung, Lennart},
	month = jul,
	year = {2015},
	pages = {1291--1296},
}

@ARTICLE{Wahlberg1994system,
  author={Wahlberg, B.},
  journal={IEEE Transactions on Automatic Control}, 
  title={System identification using Kautz models}, 
  year={1994},
  volume={39},
  number={6},
  pages={1276-1282},
  doi={10.1109/9.293196}}

@ARTICLE{Ninness1999,
  author={Ninness, B. and Hjalmarsson, H. and Gustafsson, F.},
  journal={IEEE Transactions on Automatic Control}, 
  title={The fundamental role of general orthonormal bases in system identification}, 
  year={1999},
  volume={44},
  number={7},
  pages={1384-1406},
}

@ARTICLE{Qi2017a,
  author={Qi, Tian and Chen, Jie and Su, Weizhou and Fu, Minyue},
  journal={IEEE Transactions on Automatic Control}, 
  title={Control Under Stochastic Multiplicative Uncertainties: Part I, Fundamental Conditions of Stabilizability}, 
  year={2017},
  volume={62},
  number={3},
  pages={1269-1284},
}

@ARTICLE{Qi2017b,
  author={Su, Weizhou and Chen, Jie and Fu, Minyue and Qi, Tian},
  journal={IEEE Transactions on Automatic Control}, 
  title={Control Under Stochastic Multiplicative Uncertainties: Part II, Optimal Design for Performance}, 
  year={2017},
  volume={62},
  number={3},
  pages={1285-1300},
}

@ARTICLE{Bamieh2020,
  author={Bamieh, Bassam and Filo, Maurice},
  journal={IEEE Transactions on Automatic Control}, 
  title={An Input–Output Approach to Structured Stochastic Uncertainty}, 
  year={2020},
  volume={65},
  number={12},
  pages={5012-5027},
}

@article{Jianqi2021,
  title={Mean-Square Stability and Stabilizability Analyses of LTI Systems Under Spatially Correlated Multiplicative Perturbations},
  author={Jianqi Chen and Tian Qi and Jie Chen},
  year={2021},
  journal      = {arXiv preprint arXiv:2112.05363}
}

@article{Antoulas08,
author = {Gugercin, S. and Antoulas, A. C. and Beattie, C.},
title = {$\mathcal{H}_2$ Model Reduction for Large-Scale Linear Dynamical Systems},
journal = {SIAM Journal on Matrix Analysis and Applications},
volume = {30},
number = {2},
pages = {609-638},
year = {2008},
}

@ARTICLE{Giarre97,
  author={Giarre, L. and Milanese, M.},
  journal={IEEE Transactions on Automatic Control}, 
  title={Model quality evaluation in $\mathcal{H}_2$ identification}, 
  year={1997},
  volume={42},
  number={5},
  pages={691-698},
  doi={10.1109/9.580876}}

@incollection{KIRSCH2005243,
	title        = {Chapter 6 - Transfinite Diameter, Chebyshev Constant and Capacity},
	author       = {Siegfried Kirsch},
	year         = 2005,
	booktitle    = {Geometric Function Theory},
	publisher    = {North-Holland},
	series       = {Handbook of Complex Analysis},
	volume       = 2,
	pages        = {243--308},
	doi          = {https://doi.org/10.1016/S1874-5709(05)80010-1},
	issn         = {1874-5709},
	url          = {https://www.sciencedirect.com/science/article/pii/S1874570905800101},
	editor       = {R. Kühnau}
}

@article{menke_distribution_1985,
	title        = {On the distribution of {Tsuji} points},
	author       = {Menke, Klaus},
	year         = 1985,
	month        = sep,
	journal      = {Mathematische Zeitschrift},
	volume       = 190,
	number       = 3,
	pages        = {439--446},
	doi          = {10.1007/BF01215143},
	issn         = {0025-5874, 1432-1823},
	url          = {http://link.springer.com/10.1007/BF01215143},
	urldate      = {2023-05-14},
	language     = {en}
}

@article{hachicha2014n4sid,
	title        = {N4SID and MOESP algorithms to highlight the ill-conditioning into subspace identification},
	author       = {Hachicha, Slim and Kharrat, Maher and Chaari, Abdessattar},
	year         = 2014,
	journal      = {International Journal of Automation and Computing},
	publisher    = {Springer},
	volume       = 11,
	number       = 1,
	pages        = {30--38}
}

@incollection{Ljung1998,
   author = {Lennart Ljung},
   booktitle = {Signal Analysis and Prediction},
   doi = {10.1007/978-1-4612-1768-8_11},
   pages = {163-173},
   publisher = {Birkhäuser, Boston, MA},
   title = {System Identification},
   url = {https://link.springer.com/chapter/10.1007/978-1-4612-1768-8_11},
   year = {1998},
}

@book{chen2000,
	title        = {Control-Oriented System Identification: An {$\mathcal{H}_\infty$} Approach},  
	author       = { Jie Chen and Guoxiang Gu},
	year         = {2000},
publisher = {Wiley, New York}, 
}

@article{chiuso_ill-conditioning_2004,
	title        = {On the ill-conditioning of subspace identification with inputs},
	author       = {Chiuso, Alessandro and Picci, Giorgio},
	year         = 2004,
	month        = apr,
	journal      = {Automatica},
	volume       = 40,
	number       = 4,
	pages        = {575--589},
	doi          = {10.1016/j.automatica.2003.11.009},
	issn         = {00051098},
	url          = {https://linkinghub.elsevier.com/retrieve/pii/\\S0005109803003674},
	urldate      = {2022-10-18},
	language     = {en}
}

@article{Bauer2000,
   author = {D Bauer and M Jansson},
   journal = {Automatica},
   pages = {497-509},
   title = {Analysis of the asymptotic properties of the MOESP type of subspace algorithms},
   volume = {36},
   year = {2000},
}

@article{Van1996,
   author = {Peter Van Overschee and Bart De Moor},
   doi = {10.1007/978-1-4613-0465-4},
   journal = {Subspace Identification for Linear Systems},
   publisher = {Springer US},
   title = {Subspace Identification for Linear Systems},
   year = {1996},
}

@article{ho1966effective,
	title        = {Effective construction of linear state-variable models from input/output functions},
	author       = {HO, BL and K{\'a}lm{\'a}n, Rudolf E},
	year         = 1966,
	journal      = {at-Automatisierungstechnik},
	publisher    = {Oldenbourg Wissenschaftsverlag},
	volume       = 14,
	number       = {1-12},
	pages        = {545--548}
}

@INPROCEEDINGS{tsiamis_linear_2021,
  author={Tsiamis, Anastasios and Pappas, George J.},
  booktitle={2021 60th IEEE Conference on Decision and Control (CDC)}, 
  title={Linear Systems can be Hard to Learn}, 
  year={2021},
  volume={},
  number={},
  pages={2903-2910},
  doi={10.1109/CDC45484.2021.9682778}}

@article{trefethen_numerical_2020,
	title        = {Numerical {Conformal} {Mapping} with {Rational} {Functions}},
	author       = {Trefethen, Lloyd N.},
	year         = 2020,
	month        = nov,
	journal      = {Computational Methods and Function Theory},
	volume       = 20,
	number       = {3-4},
	pages        = {369--387},
	doi          = {10.1007/s40315-020-00325-w},
	issn         = {1617-9447, 2195-3724},
	url          = {https://link.springer.com/10.1007/s40315-020-00325-w},
	urldate      = {2023-05-15},
	language     = {en}
}

@inproceedings{li2022fundamental,
	title        = {Fundamental Limit on SISO System Identification},
	author       = {Li, Jiayun and Sun, Shuai and Mo, Yilin},
	year         = 2022,
	booktitle    = {2022 IEEE 61st Conference on Decision and Control (CDC)},
	volume       = {},
	number       = {},
	pages        = {856--861},
	doi          = {10.1109/CDC51059.2022.9993203}
}

@article{Tsiamis2019,
   author = {Anastasios Tsiamis and George J. Pappas},
   doi = {10.1109/CDC40024.2019.9029499},
   isbn = {9781728113982},
   issn = {25762370},
   journal = {Proceedings of the IEEE Conference on Decision and Control},
   month = {12},
   pages = {3648-3654},
   publisher = {Institute of Electrical and Electronics Engineers Inc.},
   title = {Finite Sample Analysis of Stochastic System Identification},
   volume = {2019-December},
   year = {2019},
}

@article{oymak_revisiting_2022,
	title        = {Revisiting {Ho}–{Kalman}-{Based} {System} {Identification}: {Robustness} and {Finite}-{Sample} {Analysis}},
	shorttitle   = {Revisiting {Ho}–{Kalman}-{Based} {System} {Identification}},
	author       = {Oymak, Samet and Ozay, Necmiye},
	year         = 2022,
	month        = apr,
	journal      = {IEEE Transactions on Automatic Control},
	volume       = 67,
	number       = 4,
	pages        = {1914--1928},
	doi          = {10.1109/TAC.2021.3083651},
	issn         = {0018-9286, 1558-2523, 2334-3303},
	url          = {https://ieeexplore.ieee.org/document/9440770/},
	urldate      = {2023-03-02},
	language     = {en}
}

@article{Verhaegen1994,
   author = {Michel Verhaegen},
   doi = {10.1016/0005-1098(94)90229-1},
   issn = {0005-1098},
   issue = {1},
   journal = {Automatica},
   month = {1},
   pages = {61-74},
   publisher = {Pergamon},
   title = {Identification of the deterministic part of MIMO state space models given in innovations form from input-output data},
   volume = {30},
   year = {1994},
}

@article{mo_network_2009,
	title        = {Network {Energy} {Minimization} via {Sensor} {Selection} and {Topology} {Control}},
	author       = {Mo, Yilin and Ambrosino, Roberto and Sinopoli, Bruno},
	year         = 2009,
	month        = sep,
	journal      = {IFAC Proceedings Volumes},
	volume       = 42,
	number       = 20,
	pages        = {174--179},
	doi          = {10.3182/20090924-3-IT-4005.00030},
	issn         = 14746670,
	url          = {https://linkinghub.elsevier.com/retrieve/pii/S1474667015361553},
	urldate      = {2023-03-04},
	language     = {en}
}

@article{van_den_hof_system_1995,
	series = {Trends in {System} {Identification}},
	title = {System identification with generalized orthonormal basis functions},
	volume = {31},
	issn = {0005-1098},
	url = {https://www.sciencedirect.com/science/article/pii/0005109895000744},
	doi = {10.1016/0005-1098(95)00074-4},
	number = {12},
	urldate = {2024-01-09},
	journal = {Automatica},
	author = {Van Den Hof, Paul M. J. and S.C. Heuberger, Peter and Bokor, József},
	month = dec,
	year = {1995},
	pages = {1821--1834},
}

@article{heuberger_generalized_1995,
	title = {A generalized orthonormal basis for linear dynamical systems},
	volume = {40},
	issn = {00189286},
	url = {http://ieeexplore.ieee.org/document/376057/},
	doi = {10.1109/9.376057},
	language = {en},
	number = {3},
	urldate = {2024-01-09},
	journal = {IEEE Transactions on Automatic Control},
	author = {Heuberger, P.S.C. and Van Den Hof, P.M.J. and Bosgra, O.H.},
	month = mar,
	year = {1995},
	pages = {451--465},
}

@article{toth_asymptotically_2009,
	title = {Asymptotically optimal orthonormal basis functions for {LPV} system identification},
	volume = {45},
	issn = {00051098},
	url = {https://linkinghub.elsevier.com/retrieve/pii/S0005109809000557},
	doi = {10.1016/j.automatica.2009.01.010},
	language = {en},
	number = {6},
	urldate = {2024-01-09},
	journal = {Automatica},
	author = {Tóth, Roland and Heuberger, Peter S.C. and Van Den Hof, Paul M.J.},
	month = jun,
	year = {2009},
	pages = {1359--1370},
}

@book{van_den_hof_system_2005,
	address = {London},
	isbn = {978-1-84628-178-5},
	url = {https://doi.org/10.1007/1-84628-178-4_4},
	title = {Modelling and {Identification} with {Rational} {Orthogonal} {Basis} {Functions}},
	publisher = {Springer London},
	author = {Heuberger, Peter S.C. and Van den Hof, Paul M.J. and Wahlberg, Bo},
	year = {2005},
	doi = {10.1007/1-84628-178-4_4},
}

@inproceedings{reginato_selecting_2007,
	title = {On selecting the {MIMO} {Generalized} {Orthonormal} {Basis} {Functions} poles by using {Particle} {Swarm} {Optimization}},
	url = {https://ieeexplore.ieee.org/abstract/document/7068981},
	doi = {10.23919/ECC.2007.7068981},
	urldate = {2024-03-21},
	booktitle = {2007 {European} {Control} {Conference} ({ECC})},
	author = {Reginato, Bruno C. and Oliveira, Gustavo H. C.},
	month = jul,
	year = {2007},
	pages = {5182--5188},
}

@article{sabatini_hybrid_2000,
	title = {A hybrid genetic algorithm for estimating the optimal time scale of linear systems approximations using {Laguerre} models},
	volume = {45},
	issn = {1558-2523},
	url = {https://ieeexplore.ieee.org/document/855574},
	doi = {10.1109/9.855574},
	number = {5},
	urldate = {2024-03-21},
	journal = {IEEE Transactions on Automatic Control},
	author = {Sabatini, A.M.},
	month = may,
	year = {2000},
	pages = {1007--1011},
}

@article{mi_adaptive_2021,
	title = {Adaptive {Rational} {Orthogonal} {Basis} {Functions} for {Identification} of {Continuous}-{Time} {Systems}},
	volume = {66},
	issn = {0018-9286, 1558-2523, 2334-3303},
	url = {https://ieeexplore.ieee.org/document/9096566/},
	doi = {10.1109/TAC.2020.2995827},
	language = {en},
	number = {4},
	urldate = {2024-03-21},
	journal = {IEEE Transactions on Automatic Control},
	author = {Mi, Wen and Zheng, Wei Xing},
	month = apr,
	year = {2021},
	pages = {1809--1816},
}

@article{mi_frequency-domain_2012,
	title = {Frequency-domain identification: {An} algorithm based on an adaptive rational orthogonal system},
	volume = {48},
	issn = {0005-1098},
	shorttitle = {Frequency-domain identification},
	url = {https://www.sciencedirect.com/science/article/pii/S0005109812000982},
	doi = {10.1016/j.automatica.2012.03.002},
	number = {6},
	urldate = {2024-03-21},
	journal = {Automatica},
	author = {Mi, Wen and Qian, Tao},
	month = jun,
	year = {2012},
	pages = {1154--1162},
}

@inproceedings{darwish2015perspectives,
  title={Perspectives of orthonormal basis functions based kernels in Bayesian system identification},
  author={Darwish, Mohamed and Pillonetto, Gianluigi and T{\'o}th, Roland},
  booktitle={2015 54th IEEE Conference on Decision and Control (CDC)},
  pages={2713--2718},
  year={2015},
  organization={IEEE}
}

@article{darwish_bayesian_2017,
	series = {20th {IFAC} {World} {Congress}},
	title = {Bayesian {Frequency} {Domain} {Identification} of {LTI} {Systems} with {OBFs} {Kernels}*},
	volume = {50},
	issn = {2405-8963},
	url = {https://www.sciencedirect.com/science/article/pii/S2405896317313034},
	doi = {10.1016/j.ifacol.2017.08.845},
	number = {1},
	urldate = {2024-03-31},
	journal = {IFAC-PapersOnLine},
	author = {Darwish, Mohamed A. H. and Lataire, John and Tóth, Roland},
	month = jul,
	year = {2017},
	pages = {6238--6243},
}

@article{ninness_asymptotic_1996,
	series = {13th {World} {Congress} of {IFAC}, 1996, {San} {Francisco} {USA}, 30 {June} - 5 {July}},
	title = {Asymptotic {Analysis} of {MIMO} {System} {Estimates} by the {Use} of {Orthonormal} {Bases}},
	volume = {29},
	issn = {1474-6670},
	url = {https://www.sciencedirect.com/science/article/pii/S147466701758355X},
	doi = {10.1016/S1474-6670(17)58355-X},
	number = {1},
	urldate = {2024-01-09},
	journal = {IFAC Proceedings Volumes},
	author = {Ninness, Brett and Carlos Gómez, Juan},
	month = jun,
	year = {1996},
	pages = {4291--4296},
}

@book{Madsen2008,
author = {Henrik Madsen },
title = {Time Series Analysis},
year = {2008},
publisher = {Chapman \& Hall/CRC},
}

@article{Zadeh1962,
	title = {From circuit theory to system theory},
	volume = {50},
	journal = {Proc. IRE},
	author = {Zadeh, Lofti A.},
	year = {1962},
	pages = {856-865},
}

@inproceedings{Kalman1965,
	title = {Linear stochastic filtering-Reappraisal and outlook},
	booktitle = {Proc. Symp. System Theory Polytech},
	author = {Kalman, Rudolph E.},
	month = jan,
	year = {1965},
	pages = {197-205},
}

@book{Picci2015,
author = { A. Lindquist and G. Picci},
title = {Linear Stochastic Systems: A Geometric Approach to Modeling, Estimation and Identification},
year = {2015},
publisher = {Springer Verlag},
}

@article{semi_infinite_programming,
author = {Hettich, R. and Kortanek, K. O.},
title = {Semi-Infinite Programming: Theory, Methods, and Applications},
journal = {SIAM Review},
volume = {35},
number = {3},
pages = {380-429},
year = {1993},
doi = {10.1137/1035089},

URL = { 
    

        https://doi.org/10.1137/1035089
},
eprint = { 
    

        https://doi.org/10.1137/1035089
}
}

@phdthesis{bachnas2023advancing,
title = "Advancing Process Control using Orthonormal Basis Functions",
author = "Bachnas, \{Ahmad Alrianes\}",
note = "Proefschrift.",
year = "2023",
month = feb,
day = "16",
language = "English",
isbn = "978-90-386-5656-4",
publisher = "Eindhoven University of Technology",
type = "Phd Thesis 1 (Research TU/e / Graduation TU/e)",
school = "Electrical Engineering",
}
\bibliographystyle{ieeetr}

\begin{biography}[{\includegraphics[width=1in,height=1.25in,clip,keepaspectratio]{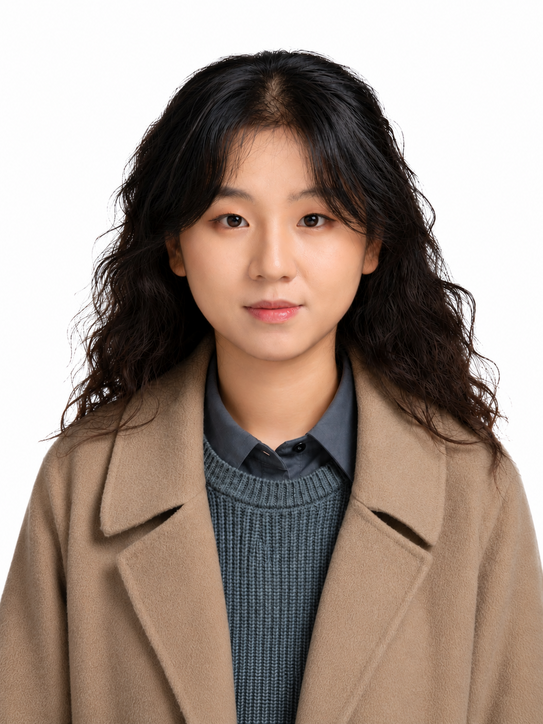}}]{Jiayun Li} received her Bachelor of Engineering degree from the Department of Automation, Tsinghua University in 2022. She is currently a Ph.D. candidate in the Department of Automation, Tsinghua University. Her research interests include system identification, learning-based control, with applications in robotics.
\end{biography}

\begin{biography}[{\includegraphics[width=1in,height=1.25in,clip,keepaspectratio]{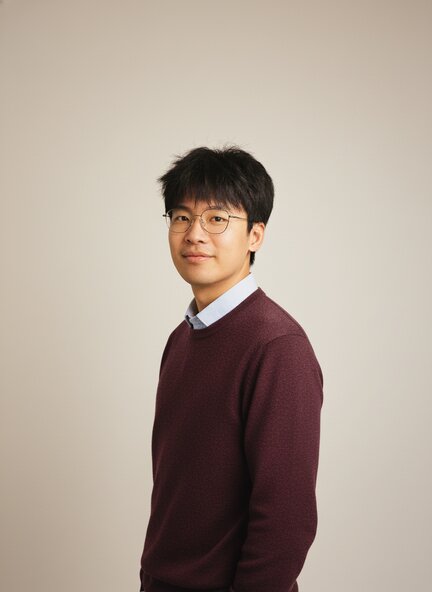}}]{Yiwen Lu}
received the B.Eng. degree from the Department of Automation, Tsinghua University, Beijing, China, in 2020, and the Ph.D. degree from the Department of Automation, Tsinghua University, in 2025. His research interests include adaptive and learning-based control, with applications in robotics.
\end{biography}
\vspace{-0.5em}
\begin{biography}[{\includegraphics[width=1in,height=1.25in,clip,keepaspectratio]{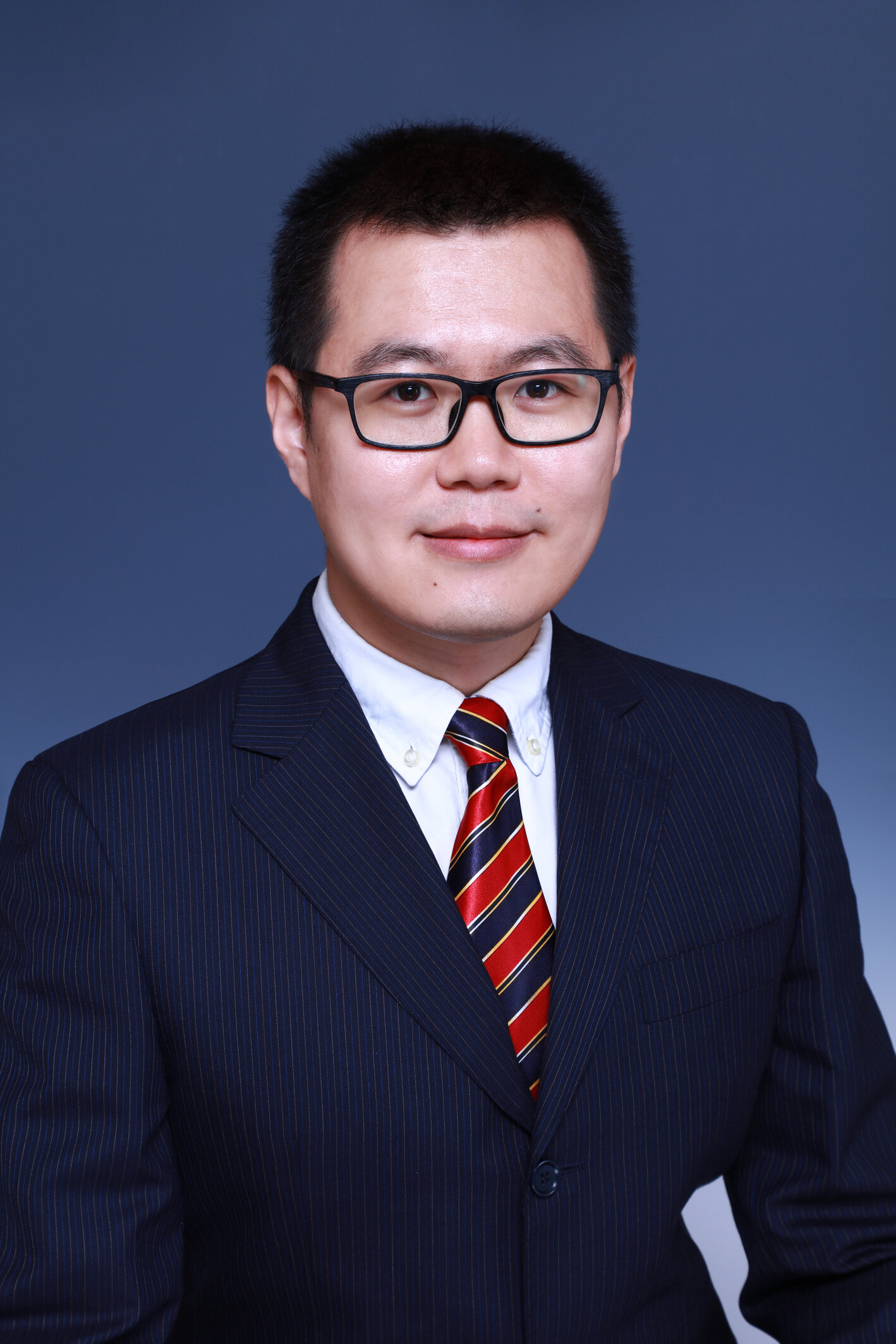}}]{Yilin Mo} is an Associate Professor in the Department of Automation, Tsinghua University. He received his Ph.D. In Electrical and Computer Engineering from Carnegie Mellon University in 2012 and his Bachelor of Engineering degree from Department of Automation, Tsinghua University in 2007. Prior to his current position, he was a postdoctoral scholar at Carnegie Mellon University in 2013 and California Institute of Technology from 2013 to 2015. He held an assistant professor position in the School of Electrical and Electronic Engineering at Nanyang Technological University from 2015 to 2018. His research interests include secure control systems and networked control systems, with applications in sensor networks and power grids.
\end{biography}
\vspace{-0.5em}
\begin{biography}[{\includegraphics[width=1in,height=1.25in,clip,keepaspectratio]{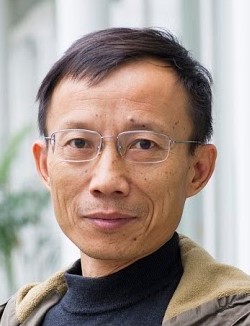}}]
	{Jie Chen}
	(S'87--M'89--SM'98--F'07) received the B.S. degree in aerospace engineering from Northwestern Polytechnic University, Xi’an, China, in 1982, the M.S.E. degree in electrical engineering, the M.A. degree in mathematics, and the Ph.D. degree in electrical engineering all from the University of Michigan, Ann Arbor, MI, USA, in 1985, 1987, and 1990, respectively.

From 1994 to 2014, he was a Professor with the University of California, Riverside, CA, USA, and was a Professor and a Chair with the Department of Electrical Engineering from 2001 to 2006. Presently, he is a Chair Professor with the Department of Electrical Engineering, City University of Hong Kong, Hong Kong. His research interests include linear multi-variable systems theory, system identification, robust control, optimization, time-delay systems, networked control, and multi-agent systems.
Dr.\ Chen is a Fellow of IEEE, AAAS, IFAC, SIAM, AAIA, and an IEEE Distinguished Lecturer.
\end{biography}

\end{document}